\documentclass[11pt,leqno]{amsart}

\usepackage{amssymb,hyperref}

\newtheorem{theorem}{Theorem}[section]
\newtheorem{lemma}[theorem]{Lemma}
\newtheorem{proposition}[theorem]{Proposition}
\newtheorem{corollary}[theorem]{Corollary}

\theoremstyle{definition}
\newtheorem{definition}[theorem]{Definition}

\newtheorem{problem}[theorem]{Problem}

\numberwithin{equation}{section}

\DeclareMathOperator{\stab}{Stab}
\DeclareMathOperator{\rstab}{Rstab}
\DeclareMathOperator{\Sym}{Sym}
\DeclareMathOperator{\Aut}{Aut}
\DeclareMathOperator{\Fr}{Fr}
\DeclareMathOperator{\charac}{char}

\renewcommand{\epsilon}{\varepsilon}

\title[Maximal subgroups of multi-edge spinal groups]{Maximal
  subgroups of \\ multi-edge spinal groups}

\author[T. Alexoudas]{Theofanis Alexoudas} \address{Theofanis
  Alexoudas: Department of Mathematics, Royal Holloway, University of
  London, Egham TW20 0EX, UK}
\email{Theofanis.Alexoudas.2009@live.rhul.ac.uk}

\author[B. Klopsch]{Benjamin Klopsch} \address{Benjamin Klopsch:
  Mathematisches Institut, Heinrich-Heine-Universit\"at, 40225
  D\"usseldorf, Germany} \email{klopsch@math.uni-duesseldorf.de}

\author[A. Thillaisundaram]{Anitha Thillaisundaram} \address{Anitha
  Thillaisundaram: Mathematisches Institut,
  Heinrich-Heine-Universit\"at, 40225 D\"usseldorf, Germany}
\email{anitha.t@cantab.net}


 
\begin{document}

\begin{abstract}
  A multi-edge spinal group is a subgroup of the automorphism group of
  a regular $p$-adic rooted tree, generated by one rooted automorphism
  and a finite number of directed automorphisms sharing a common
  directing path.  We prove that torsion multi-edge spinal groups do
  not have maximal subgroups of infinite index.  This generalizes a
  result of Pervova for GGS-groups.
\end{abstract}

\maketitle


\section{Introduction}
Branch groups are groups acting spherically transitively on a
spherically homogeneous infinite rooted tree and having subnormal
subgroup structure similar to the corresponding structure in the full
group of automorphisms of the tree. Early constructions were produced
by Grigorchuk~\cite{Grigorchuk} and Gupta and Sidki~\cite{Gupta}, and
they were generalized to so-called GGS-groups. The class of branch
groups provides important and easily describable examples for
finitely generated groups of intermediate word growth, or for finitely
generated infinite torsion groups; cf.\ the General Burnside Problem.

We deal here with multi-edge spinal groups acting on the regular
$p$-adic rooted tree $T$, for $p$ an odd prime. A multi-edge spinal
group $G = \langle a, b_1, \ldots, b_r \rangle$ is generated by a
rooted automorphism $a$ and a finite number of directed automorphisms
$b_1,\ldots,b_r$, for some $r \in \{1,\ldots,p-1\}$; see
Section~\ref{sec:mul-edge-spinal} for details. These groups generalize
GGS-groups, which correspond to the special case $r=1$.  In
particular, every torsion multi-edge spinal group is an infinite
$p$-group.  Moreover we show that, apart from one possible exception,
multi-edge spinal groups are branch; see
Proposition~\ref{prop:branch}.

Pervova~\cite{Pervova3, Pervova4} proved that the Grigorchuk group and
torsion GGS-groups do not contain maximal subgroups of infinite index.
Equivalently, these groups do not contain proper dense subgroups with
respect to the profinite topology.  On the other hand -- prompted by a
question of Grigorchuk, Bartholdi and
{\v{S}}uni{\'k}~\cite{BarthGrigSunik} -- Bondarenko gave
in~\cite{Bondarenko} a non-constructive example of a finitely
generated branch group that does have maximal subgroups of infinite
index.  Hence we face the following problem.

\begin{problem}
  Characterize among finitely generated branch groups those that
  possess maximal subgroups of infinite index and those that do not.
\end{problem}

In particular, Bondarenko's method -- by itself -- does not apply to
groups acting on the regular $p$-adic rooted tree~$T$ that are
residually finite-$p$.  It is natural to test how far Pervova's
results in~\cite{Pervova4} can be extended and multi-edge spinal
groups form a suitable generalization of GGS-groups.

\begin{theorem} \label{start} Let $G$ be a multi-edge spinal group
  acting on the regular $p$-adic rooted tree, for $p$ an odd prime,
  and suppose that $G$ is torsion. Then every maximal subgroup of $G$
  is normal of finite index~$p$.
\end{theorem}

As indicated in ~\cite{Pervova4}, one motivation for our investigation
comes from a conjecture of Passman concerning the group algebra $K[G]$
of a finitely generated group $G$ over a field $K$ with $\charac K =
p$.  The conjecture states that, if the Jacobson radical
$\mathcal{J}(K[G])$ coincides with the augmentation ideal
$\mathcal{A}(K[G])$ then $G$ is a finite $p$-group;
see~\cite[Conjecture~6.1]{Passman2}.  In~\cite{Passman2}, Passman
showed that if $\mathcal{J}(K[G])=\mathcal{A}(K[G])$ then $G$ is a
$p$-group and every maximal subgroup of $G$ is normal of index~$p$.
Hence multi-edge spinal groups that are torsion yield natural
candidates for testing Passman's conjecture.  It is important to widen
this class of candidates, as even the Gupta-Sidki group for $p=3$ does
not satisfy $\mathcal{J}(K[G]) = \mathcal{A}(K[G])$; this follows from
\cite{Said}.

Grigorchuk and Wilson~\cite{GrigWils} have generalized Pervova's
results in~\cite{Pervova3,Pervova4} by means of commensurability.  Two
groups are said to be \textit{commensurable} if they have isomorphic
subgroups of finite index.  Let $G$ be as in Theorem~\ref{start}. From
\cite[Lemma~1]{GrigWils}, it follows that whenever $G$ is a subgroup
of finite index in a group~$H$ then $H$ does not have maximal
subgroups of infinite index.  Using this and \cite[Lemma~3]{GrigWils}
we derive the following consequence of Theorem~\ref{start}.

\begin{corollary} \label{cor:main}
  Let $G$ be as in Theorem~\textup{\ref{start}}.  If $H$ is a group
  commensurable with $G$ then every maximal subgroup of $H$ has finite
  index in~$H$.
\end{corollary}

Finally, we remark that most parts
of the proof of Theorem~\ref{start} go through under the assumption
that the group is just infinite and not necessarily torsion.  One may
therefore speculate that, in fact, every just infinite multi-edge
spinal group has the property that all its maximal subgroups are of
finite index.


\section{Preliminaries} \label{sec:2} In the present section we recall
the notion of branch groups and establish prerequisites for the rest
of the paper.  For more details,
see~\cite{BarthGrigSunik,NewHorizons}.

\subsection{The regular $p$-adic rooted tree and its automorphisms}
Let $T$ the regular $p$-adic rooted tree, for an odd prime~$p$.  Let
$X$ be an alphabet on $p$ letters, e.g., $X=\{1,2,\ldots,p\}$.  The
set of vertices of $T$ can be identified with the free monoid
$\overline{X}$; in particular, the root of~$T$ corresponds to the
empty word~$\varnothing$. For each word $v\in \overline{X}$ and letter
$x$, an edge connects $v$ to $vx$. There is a natural length function
on $\overline{X}$, and the words $w$ of length $|w| = n$, representing
vertices that are at distance $n$ from the root, constitute the
\textit{$n$th layer} of the tree. The tree is called \textit{regular}
because all vertices have the same out-degree~$p$, and the
\textit{boundary} $\partial T$ consisting of all infinite rooted paths
is in one-to-one correspondence with the $p$-adic integers.  More
generally, one considers rooted trees that are not necessarily
regular, but spherically homogeneous, meaning that vertices of the
same length have the same degree.

We write $T_u$ for the full rooted subtree of $T$ that has its root at
a vertex~$u$ and includes all vertices $v$ with $u$ a prefix of
$v$. As $T$ is regular $p$-adic, for any two vertices $u$ and $v$ the
subtrees $T_u$ and $T_v$ are isomorphic under the map that deletes the
prefix $u$ and replaces it by the prefix $v$. We refer to this
identification as the natural identification of subtrees and write
$T_n$ to denote the subtree rooted at a generic vertex of level $n$.

We observe that every automorphism of $T$ fixes the root and that the
orbits of $\Aut(T)$ on the vertices of the tree $T$ are precisely its
layers.  Consider an automorphism $f \in \Aut(T)$.  The image of a
vertex $u$ under $f$ is denoted by $u^f$.  For a vertex $u$, thought
of as a word over $X$, and a letter $x \in X$ we have $(ux)^f=u^fx'$
where $x' \in X$ is uniquely determined by $u$ and~$f$.  This induces
a permutation $f(u)$ of $X$ so that
\[
(ux)^f = u^f x^{f(u)}.
\]
The automorphism $f$ is called \textit{rooted} if $f(u)=1$ for $u\ne
\varnothing$.  It is called \textit{directed}, with directing path
$\ell \in \partial T$, if the support $\{u \mid f(u)\ne1 \}$ of its
labelling is infinite and contains only vertices at distance $1$ from
$\ell$.

The \textit{section} of $f$ at a vertex $u$ is the unique automorphism
$f_u$ of $T \cong T_{|u|}$ given by the condition $(uv)^f = u^f
v^{f_u}$ for $v \in \overline{X}$.

\subsection{Subgroups of $\Aut(T)$}
Let $G$ be a subgroup of $\Aut(T)$ acting \textit{spherically
  transitively}, i.e., transitively on every layer of $T$. The
\textit{vertex stabilizer} $\stab_G(u)$ is the subgroup consisting of
elements in $G$ that fix the vertex~$u$.  For $n \in \mathbb{N}$, the
\textit{$n$th level stabilizer} $\stab_G(n)= \cap_{|v|=n} \stab_G(v)$
is the subgroup consisting of automorphisms that fix all vertices at
level $n$.  Note that elements in $\stab_G(n)$ fix all vertices up to
level $n$ and that $\stab_G(n)$ has finite index in $G$.

The full automorphism group $\Aut(T)$ is a profinite group.  Indeed,
\[
\Aut(T)= \varprojlim_{n\to\infty} \Aut(T_{[n]}),
\]
where $T_{[n]}$ denotes the subtree of $T$ on the finitely many
vertices up to level~$n$. The topology of $\Aut(T)$ is defined by the
open subgroups $\stab_{\Aut(T)}(n)$, $n \in \mathbb{N}$.  The level
stabilizers $\stab_G(n)$, $n \in \mathbb{N}$, form a natural family of
principal congruence subgroups for~$G$.  The subgroup $G$ of $\Aut(T)$
has the \textit{congruence subgroup property} if the profinite
topology and the congruence topology on $G$ coincide, i.e., if for
every subgroup $H$ of finite index in $G$, there exists some $n$ such
that $\stab_G(n)\subseteq H$.

Every $g\in \stab_{\Aut(T)} (n)$ can be identified with a collection
$g_1,\ldots,g_{p^n}$ of elements of $\Aut(T_n)$, where $p^n$ is the
number of vertices at level $n$.  Indeed, denoting by $u_1, \ldots,
u_{p^n}$ the vertices of $T$ at level $n$, there is a natural
isomorphism
\[
\stab_{\Aut(T)}(n) \cong \prod\nolimits_{i=1}^{p^n} \Aut(T_{u_i})
\cong \Aut(T_n) \times \overset{p^n}{\ldots} \times \Aut(T_n).
\]
Since $T$ is regular, $\Aut(T_n)$ is isomorphic to $\Aut(T)$ after the
natural identification of subtrees. Therefore the decomposition
$g=(g_1,\ldots,g_{p^n})$ defines an embedding
\[
\psi_n \colon \stab_{\Aut(T)}(n) \rightarrow \prod\nolimits_{i=1}^{p^n}
\Aut(T_{u_i}) \cong \Aut(T) \times \overset{p^n}{\ldots} \times
\Aut(T).
\]

We write $U_u^G$ for the restriction of the vertex stabilizer
$\stab_G(u)$ to the subtree rooted at a vertex $u$. Since $G$ acts
spherically transitively, the vertex stabilizers at every level are
conjugate under~$G$. We write $U_n^G$ for the common isomorphism type
of the restriction of the $n$th level vertex stabilizers, and we call
it the \textit{$n$th upper companion group} of $G$.  We say that $G$
is \textit{fractal} if every upper companion group $U_n^G$ coincides
with the group~$G$, after the natural identification of subtrees.

Next, the subgroup $\rstab_G(u)$, consisting of all automorphisms in
$G$ that fix all vertices $v$ of $T$ not having $u$ as a prefix, is
called the \textit{rigid vertex stabilizer} of $u$ in $G$. For a
vertex $u\in T$, we write $\rstab_G(u)_u$ for the restriction of the
rigid vertex stabilizer to the subtree rooted at $u$. The
\textit{rigid $n$th level stabilizer} is the product
\[
\rstab_G(n)=\prod\nolimits_{i=1}^{p^n} \rstab_G(u_i) \trianglelefteq G
\]
of the rigid vertex stabilizers of the vertices $u_1, \ldots, u_{p^n}$
at level~$n$. Since $G$ acts spherically transitively, the rigid
vertex stabilizers at each level are conjugate under~$G$.  The common
isomorphism type $L_n^G$ of the $n$th level rigid vertex stabilizers
is called the \textit{$n$th lower companion group} of $G$.

\subsection{Branch groups} More generally, we recall that a
\textit{spherically homogeneous} infinite rooted tree $T =
T_{\overline{m}}$ is constructed over a sequence of alphabets
$X_1,X_2,\ldots$ with $|X_i|=m_i \ge 2$, where $\overline{m} =
(m_n)_{n=1}^\infty$ is a sequence of natural numbers, in such a way
that all vertices at the same level $n-1$ have the same
out-degree~$m_n$. In the case of a regular $p$-adic rooted tree, for
$p$ an odd prime, the branching sequence is constant:
$\overline{m}=(p,p,\ldots)$.

\begin{definition}[\cite{NewHorizons}] \label{branch} A group $G$ is a
  \textit{branch group}, if there is a spherically homogeneous rooted
  tree $T=T_{\overline{m}}$, with branching sequence
  $\overline{m}=(m_n)_{n=1}^\infty$, and an embedding $G\hookrightarrow
  \Aut(T)$ such that
  \begin{enumerate}
  \item[(1)] the group $G$ acts transitively on each layer of the
    tree;
  \item[(2)] for each level $n$ there exists a subgroup $L_n$ of the
  automorphism group $\Aut(T_n)$ of the full subtree $T_n$ rooted at a
  level $n$ vertex such that the direct product
  \[
  H_n =L_n^{(1)}\times \ldots \times L_n^{(N_n)} \le
  \stab_{\Aut(T)}(n), \text{ where } L_n^{(j)}\cong L_n,
  \]
  of $N_n= \prod_{i=1}^n m_i$ copies of $L_n$ is normal and of finite
  index in $G$.
  \end{enumerate}
\end{definition}

There exists an alternative and more intrinsic algebraic definition,
which can also be found in~\cite{NewHorizons}. For a fixed embedding
$G\hookrightarrow\Aut(T)$, the pair
$((L_n)_{n=1}^\infty,(H_n)_{n=1}^\infty)$ is called a \textit{branch
  structure}.  If $G$ is branch then a `standard' branch structure is
given by $((L_n^G)_{n=1}^\infty,(\rstab_G(n))_{n=1}^\infty)$. Thus
condition (2) of the definition means that all rigid level stabilizers
$\rstab_G(n)$ are of finite index in $G$.


\section{Multi-edge spinal groups} \label{sec:mul-edge-spinal} Let $T$
be the regular $p$-adic rooted tree, for an odd prime $p$.  The
vertices of $T$ can be identified with words over an alphabet $X$ of
size $p$; sometimes it is convenient to label them explicitly by
finite sequences in $\{1,\ldots,p\}$.  Let $\ell =(l_n)_{n=0}^\infty$
be an infinite path in $T$ starting at the root, with $l_n=x_1 \cdots
x_n$ where $x_1,\ldots,x_n\in X$. For every $n \ge 1$ and $y \in X
\setminus \{x_n\}$, we denote by $s_{n,y}$ the immediate descendants
of $l_{n-1}$ not lying in $\ell$. The doubly indexed family
$S=(s_{n,y})_{n,y}$ is a \textit{multi-edge spine} in $T$.  Recall
that $\Aut(T)$ acts transitively within the layers of $T$, and in
particular, on the boundary $\partial T$. Hence, conjugating by an
element of $\Aut(T)$, we may choose, for simplicity, the spine to be
associated to the rightmost infinite path $(\varnothing, u_p, u_{pp},
\ldots)$ starting at the root vertex of the tree.

\subsection{Construction of multi-edge spinal groups}
By $a$ we denote the rooted automorphism, corresponding to the
$p$-cycle $(1 \, 2 \, \ldots \, p)\in \Sym(p)$, that cyclically
permutes the vertices $u_1, \ldots, u_p$ of the first level.  Recall
the coordinate map
\[
\psi_1\colon \stab_{\Aut(T)}(1) \rightarrow \Aut(T_{u_1}) \times
\ldots \times \Aut(T_{u_p}) \cong \Aut(T) \times \overset{p}{\ldots}
\times \Aut(T).
\]
Given $r \in \mathbb{N}$ and a finite $r$-tuple $\mathbf{E}$ of
$(\mathbb{Z}/p\mathbb{Z})$-linearly independent vectors
\[
\mathbf{e}_i =(e_{i,1}, e_{i,2},\ldots , e_{i,p-1})\in
(\mathbb{Z}/p\mathbb{Z})^{p-1}, \qquad i\in \{1,\ldots,r \},
\]
we recursively define directed automorphisms $b_1, \ldots, b_r \in
\stab_{\Aut(T)}(1)$ via
\[
\psi_1(b_i)=(a^{e_{i,1}}, a^{e_{i,2}},\ldots,a^{e_{i,p-1}},b_i),
\qquad i\in \{1,\ldots,r \}.
\]
We call the subgroup $G=G_{\mathbf{E}}=\langle a, b_1, \ldots, b_r
\rangle$ of $\Aut(T)$ the \textit{multi-edge spinal group} associated
to the defining vectors $\mathbf{E}$.  We observe that $\langle a
\rangle \cong C_p$ and $\langle b_1, \ldots, b_r \rangle \cong C_p^r$
are elementary abelian $p$-groups.

By choosing only one vector 
\[
\mathbf{e}=(e_1,\ldots,e_{p-1})\in (\mathbb{Z}/p\mathbb{Z})^{p-1}
\]
and defining an automorphism $b$ of $\Aut(T)$ via
\[
\psi_1(b) =(a^{e_1},\ldots , a^{e_{p-1}},b)
\]
we obtain the GGS-group $G_{(\mathbf{e})}=\langle a,b\rangle$
corresponding to the defining vector~$\mathbf{e}$.  For instance, the
Gupta-Sidki group for the prime $p$ arises by choosing
$\mathbf{e}=(1,-1,0,\ldots,0)$; see~\cite{Gupta}.

\subsection{General properties of multi-edge spinal groups}
The proof of the following result is straightforward; details may be
found in~\cite{thesis}.

\begin{proposition} \label{transitive} Let $G$ be a multi-edge spinal
  group. Then every section of every element of $G$ is contained in
  $G$. Moreover, $G$ acts spherically transitively on the tree $T$ and
  $G$ is fractal.
\end{proposition}

The next theorem, adapted to the present context, gives necessary and
sufficient conditions for a multi-edge spinal group to be periodic.

\begin{theorem}[\cite{NewHorizons,Vovkivsky}] \label{torsion} Let
  $G_{\mathbf{E}}=\langle a, b_1, \ldots, b_r \rangle$ be a multi-edge
  spinal group corresponding to an $r$-tuple of defining vectors
  $\mathbf{E} = (\mathbf{e}_i)_{i=1}^r$. Then $G$ is an infinite
  $p$-group if and only if for every $\mathbf{e}_i = (e_{i,1}, \ldots,
  e_{i,p-1})$,
  \[
  \sum\nolimits_{j=1}^{p-1} e_{i,j}\equiv 0 \pmod{p}.
  \]
\end{theorem}

The next lemma shows: by a `change of coordinates', we can arrange
that $e_{1,1}=1$ in the defining vector $\mathbf{e}_1$ of a multi-edge
spinal group $G_\mathbf{E}$.

\begin{lemma}\label{dagger}
  Let $G = G_\mathbf{E}$ be a multi-edge spinal group. Then there
  exists an automorphism $f\in \Aut(T)$ of the form $f=f_0f_1=f_1f_0$,
  where $f_0$ is a rooted automorphism corresponding to a permutation
  $\pi \in \Sym(p)$ and $f_1\in \stab_G(1)$ with
  $\psi_1(f_1)=(f,\ldots,f)$, such that $G^f =
  G_{\widetilde{\mathbf{E}}} =\langle a, \tilde{b}_1, \ldots,
  \tilde{b}_r \rangle$ is a multi-edge spinal group generated by the
  rooted automorphism $a$ and directed automorphisms $\tilde{b}_1
  =b_1^f,\ldots, \tilde{b}_r =b_r^f$ satisfying $\psi_1(\tilde{b}_i)
  =(a^{\tilde{e}_{i,1}},\ldots,a^{\tilde{e}_{i,p-1}},\tilde{b}_i)$ for
  $i\in \{1, \ldots,r \}$ with $\tilde{e}_{1,1}=1$.
\end{lemma}

\begin{proof}
  Since the defining vectors $\mathbf{e}_1,\ldots, \mathbf{e}_r$ for
  $G$ are linearly independent over $\mathbb{Z}/p\mathbb{Z}$, each
  $\mathbf{e}_i$ satisfies
  \[
  \mathbf{e}_i=(e_{i,1},\ldots ,e_{i,p-1})\not \equiv \mathbf{0}
  \pmod{p}.
  \]
  In particular $\mathbf{e}_1 \not \equiv \mathbf{0}
  \pmod{p}$. Without loss of generality, assume that $e_{1,k}\equiv k$
  for some $k\in \{1,\ldots, p-1\}$; otherwise we replace $b_1$ by a
  power of itself. Then there exists some $l\in \{1,\ldots,p-1\}$ such
  that $kl\equiv 1 \pmod{p}$. Define a permutation $\pi \in \Sym(p)$
  by $x\pi = lx$, where $x\in \{1,\ldots, p\}$ represents a vertex in
  the first level of the tree $T$. Observe that $x\pi^{-1}=kx$ for all
  $x\in \{1,\ldots,p\}$. We consider the automorphism $f =f_0 f_1 =
  f_1 f_0 \in \Aut(T)$, where $f_0$ is a rooted automorphism
  corresponding to the permutation $\pi \in \Sym(p)$ and $f_1 \in
  \stab_G(1)$ is given by $\psi_1(f_1)=(f,\ldots,f)$.

  Set $\tilde{a}=(a^k)^f=(a^k)^{f_0}$. Then, for all $x\in
  \{1,\ldots,p\}$,
  \[
  x^{\tilde{a}} \equiv x^{f_0^{-1}a^kf_0} \equiv (kx)^{a^kf_0} \equiv
  (kx+k)^{f_0} \equiv (kx+k)l \equiv x+1 \equiv x^a \pmod{p}.
  \]
  Hence $\tilde{a}=a$. It follows that $a=(a^k)^f=(a^f)^k$, implying
  $a^f=a^l$.

  Setting $\tilde{b}_i=(b_i)^f$ for $i \in \{1,\ldots,r\}$, we obtain
  \begin{align*}
    \tilde{b}_i & = (f^{-1},\ldots,f^{-1}) (a^{e_{i,k}}, \ldots,
    \underbrace{a^{e_{i,1}}}_\textrm{$l^{\text{th}}$ resp.},
    \underbrace{a^{e_{i,k+1}}}_\textrm{$(l+1)^{\text{th}}$ coord.},
    \ldots, a^{e_{i,p-k}},b_i) (f,\ldots,f) \\
    & =(a^{le_{i,k}}, \ldots, a^{le_{i,p-k}}, \tilde{b}_i).
  \end{align*}
  In particular, $\tilde{e}_{1,1}\equiv le_{1,k}\equiv lk\equiv 1
  \pmod{p}$.  Thus $f$ as defined above has the required properties.
\end{proof}

In preparation for Proposition~\ref{Amaia-gamma} we establish the
following lemma.

\begin{lemma} \label{linearalgebra} Let $G_{\mathbf{E}}= \langle
  a,b_1,\ldots,b_r \rangle$ be a multi-edge spinal group associated to
  an $r$-tuple $\mathbf{E}$ with $r \ge 2$. Then there exists an
  $r$-tuple of defining vectors $\widetilde{\mathbf{E}}$ such that
  $G_{\widetilde{\mathbf{E}}}$ is conjugate to $G_{\mathbf{E}}$ by an
  element $f\in \Aut(T)$ as in Lemma~\textup{\ref{dagger}} and the
  following hold:
  \begin{enumerate}
  \item[(1)] $\tilde{e}_{i,1} \equiv 1 \pmod{p}$ for each $i\in
    \{1,\ldots,r\}$;
  \item[(2)] if $r=2$ and $p=3$, then $\tilde{\mathbf{e}}_1=(1,0),
    \tilde{\mathbf{e}}_2=(1,1)$;
  \item[(3)] if $r=2$ and $p>3$, then either
    \begin{enumerate}
    \item[(a)] for each $i\in \{1,2\}$ there exists $k\in \{2,\ldots,
      p-2\}$ such that $\tilde{e}_{i,k-1} \tilde{e}_{i,k+1} \not
      \equiv \tilde{e}_{i,k}^2 \pmod{p}$, or
    \item[(b)] $\tilde{\mathbf{e}}_1 = (1,0,\ldots,0,0),
      \tilde{\mathbf{e}}_2 = (1,0,\ldots,0,1)$;
    \end{enumerate}
  \item[(4)] if $r\ge 3$ then for each $i\in \{1,\ldots,r\}$ there
    exists $k\in \{2,\ldots, p-2\}$ such that $\tilde{e}_{i,k-1}
    \tilde{e}_{i,k+1} \not \equiv \tilde{e}_{i,k}^2 \pmod{p}$.
  \end{enumerate}
\end{lemma}

\begin{proof}
  We split the proof into two cases: $r\ge 3$ and $r=2$.

  \smallskip

  \noindent \underline{Case 1}: $r \ge 3$.  Observe that $p\ge 5$ and
  consider the $r \times (p-1)$-matrix
  \[
  M(\mathbf{E})= \left( \begin{array}{ccc}
      e_{1,1} & \ldots & e_{1,p-1} \\
      e_{2,1} & \ldots & e_{2,p-1} \\
      \vdots & \ddots & \vdots \\
      e_{r,1} & \ldots & e_{r,p-1} \end{array} \right)
  \]
  encoding the defining vectors for the group $G_\mathbf{E}$.  By
  Lemma~\ref{dagger}, we may assume that $e_{1,1}\not \equiv 0$.
  Using elementary row operations, we transform $M(\mathbf{E})$ into
  reduced row-echelon form:
  \[
  \left( \begin{array}{ccccccccccccc}
      1      & a_1      & \ldots & a_m       & 0 & *       & \ldots & *       & 0 & *      & \ldots & * \\
      0      & 0 & \ldots & 0       & 1 & *       & \ldots & *       & 0 & *      & \ldots & * \\
      \vdots & \vdots &        & \vdots  &   & \vdots  &        & \vdots  &   & \vdots &        & \vdots \\
      0 & 0 & \ldots & 0 & 0 & 0 & \ldots & 0 & 1 & * & \ldots &
      * \end{array} \right),
  \]
  where $m \geq 0$, $a_1, \ldots, a_m \in \mathbb{Z}/p\mathbb{Z}$ and
  the symbols $*$ denote other, unspecified elements of
  $\mathbb{Z}/p\mathbb{Z}$.  Adding the 1st row to all other rows, we
  obtain
  \begin{equation}\label{eq:matrix}
    M(\widetilde{\mathbf{E}})=\left( \begin{array}{ccccccccccccc}
        1      & a_1      & \ldots & a_m      & 0 & *      & \ldots & *      & 0 & *      & \ldots & * \\
        1      & a_1      & \ldots & a_m      & 1 & *      & \ldots & *      & 0 & *      & \ldots & * \\
        \vdots & \vdots &        & \vdots &   & \vdots &        & \vdots &   & \vdots &        & \vdots \\
        1      & a_1      & \ldots & a_m      & 0 & *      & \ldots & *      & 1 & *      & \ldots & * \end{array} \right).
  \end{equation}
  The row operations that we carried out yield a new set of generators
  for $\langle b_1, \ldots, b_r \rangle$, corresponding to an
  $r$-tuple $\widetilde{\mathbf{E}}$ of defining vectors that are
  encoded in the rows of~$M(\widetilde{\mathbf{E}})$.

  Let $i \in \{1,\ldots, r\}$ and consider the $i$th row
  of~$M(\widetilde{\mathbf{E}})$.  We identify two patterns which
  guarantee that the $i$th row satisfies the condition in~(4):
  \begin{align*}
    \text{(A)} & \qquad (* \,\, \ldots \,\, * \,\, x \,\, y \,\, 0
    \,\, * \,\, \ldots \,\, *), \\
    \text{(B}) & \qquad (* \,\, \ldots \,\, * \,\, 0 \,\, y \,\, x
    \,\, * \,\, \ldots \,\, *),
  \end{align*}
  where $x,y \in \mathbb{Z}/p\mathbb{Z}$ with $y\not \equiv 0$ and the
  symbols $*$ again denote unspecified elements.  Observe that, if the
  patterns (A) and (B) do not appear in the $i$th row, then the row
  does not have any zero entries at all or must be of the form $(* \,\, 0
  \,\, \ldots \,\, 0 \,\, *)$.

  Suppose first that $2 \le i \le r-1$. In this case the $i$th row
  contains at least one zero entry and cannot be of the form $(* \,\,
  0 \,\, \ldots \,\, 0 \,\, *)$.  Hence the pattern (A) or (B) occurs.

  Next suppose that $i = r$ and assume that patterns (A) or (B) do not
  appear.  As $r \ge 3$ the $r$th row contains at least one zero entry
  and consequently has the form $(1 \,\, 0 \,\, \ldots \,\, 0 \,\,
  1)$.  Changing generators, we may replace the $r$th row by the $r$th
  row minus the $2$nd row plus the $1$st row, yielding
  \begin{equation} \label{eq:r} \left( \begin{array}{ccccccccc} 1 & 0
        & \ldots & 0 & -1 & * & \ldots & * &1 \end{array} \right)
  \end{equation}
  with $m$ zeros between the entries $1$ and $-1$.  If $m>0$ then
  pattern (B) occurs in this new row.  Suppose that $m=0$.  Then the
  row takes the form
  \begin{equation}\label{eq:contra}
    \left( \begin{array}{cccccc}
        1 & -1 & * & \ldots & * & 1  \end{array} \right).
  \end{equation}
  For the condition in (4) to fail, we would need the row to be equal
  to $(1 \,\, -1 \,\, \phantom{-}1 \,\, -1 \,\, \ldots \,\, 1 \,\,
  -1)$ with the final entry being $-1$ as $p-1$ is even. This
  contradicts~\eqref{eq:contra}.

  Finally, suppose that $i=1$.  Similarly as above, we assume that
  patterns (A) and (B) do not occur.  Since it contains at least one
  zero entry, the $1$st row is of the form
  \begin{equation} 
    \left( \begin{array}{ccccc}
        1 & 0 & \ldots & 0 & * \end{array} \right)
  \end{equation}
  and we change generators as follows.  Generically, we replace the
  $1$st row by the $1$st row plus the $2$nd row minus the $3$rd row.
  Only if $r=3$ and we already changed the $r$th row as described
  above, we replace the $1$st row by $2$ times the $1$st row minus the
  $3$rd row.  In any case, this gives a new $1$st row:
  \[
  \left( \begin{array}{cccccccccccc} 1 & 0 & \ldots & 0 & 1 & * &
      \ldots & * & -1 & * & \ldots & * \end{array} \right)
  \]
  with $m$ zeros between the entries $1$ and~$1$.  If $m >0$ then
  pattern (B) occurs.  Suppose that $m=0$ so that the new row takes
  the form
  \begin{equation} \label{eq:contra2} \left( \begin{array}{ccccccccc}
        1 & 1 & * & \ldots & * & -1 & * & \ldots & * \end{array}
    \right).
  \end{equation}
  For the condition in~(4) to fail, the row would have to be of the
  form $(1 \,\, 1 \,\, \ldots \,\, 1)$
  contradicting~\eqref{eq:contra2}.

  \smallskip

  \noindent \underline{Case 2}: $r = 2$.  The statement in (2) for
  $p=3$ can clearly be achieved by a simple change of generators.  Now
  we suppose that $p>3$.  By Lemma~\ref{dagger}, we may assume that
  $e_{1,1}\not \equiv 0$.  Using elementary row operations, we
  transform the $2 \times (p-1)$-matrix $M(\mathbf{E})$ encoding the
  defining vectors into reduced row-echelon form
  \[
  \left( \begin{array}{cccc}
      1 & \mathbf{a}  & 0 & \mathbf{b} \\
      0 & \mathbf{0} & 1 & \mathbf{c} \end{array} \right),
  \]
  where at most one of $\left( \begin{array}{c} \mathbf{a} \\
      \mathbf{0} \end{array} \right)$ or $\left( \begin{array}{c}
      \mathbf{b} \\ \mathbf{c} \end{array} \right)$ could be the empty
  matrix.  Further row operations, corresponding to multiplication on
  the left by $\left( \begin{array}{cc} 1 & y \\ 1 & z \end{array}
  \right)$, where $y,z \in \mathbb{Z}/p\mathbb{Z}$ with $y \not \equiv
  z$ are to be specified below, yield
  \begin{equation} \label{eq:lambda} M(\widetilde{\mathbf{E}})=
    \left( \begin{array}{cccc}
        1 & \mathbf{a}  & y & \mathbf{b}+y\mathbf{c} \\
        1 & \mathbf{a} & z & \mathbf{b}+z\mathbf{c} \end{array}
    \right)
  \end{equation}
  encoding an $r$-tuple $\widetilde{\mathbf{E}}$ of defining vectors
  for a new set of generators.  

  First suppose that $\mathbf{a} = \mathbf{0}$ is not empty and zero.
  If $\mathbf{b} = ()$ then
  \[
  M(\widetilde{\mathbf{E}})= \left( \begin{array}{ccccc}
      1 & 0 & \ldots & 0 & y \\
      1 & 0 & \ldots & 0 & z \end{array} \right),
  \]
  leads to~(3)(b). Otherwise, if $\mathbf{b} \ne ()$, we choose $y
  \equiv 1$ and $z \equiv -1 \pmod{p}$, yielding pattern (B) in both
  rows so that the condition in~(3)(a) holds.

  Next suppose that $\mathbf{a} = (a_1 \,\, \ldots \,\, a_m) \ne
  \mathbf{0}$ is not empty and non-zero.  Suppose further that the
  truncated rows $(1 \,\, \mathbf{a} \,\, y)$, $(1 \,\, \mathbf{a}
  \,\, z)$ do not yet satisfy the condition in~(3)(a).  Then pattern
  (B) does not occur in these and $\mathbf{a}$ cannot have any zero
  entries.  Consequently, there exists $\lambda\in
  \mathbb{Z}/p\mathbb{Z} \setminus \{0\}$ such that
  $M(\widetilde{\mathbf{E}})$ is of the form
  \[
  M(\widetilde{\mathbf{E}})= \left( \begin{array}{ccccccccc} 1 &
      \lambda & \lambda ^2 & \ldots & \lambda ^m & y & * & \ldots
      & * \\
      1 & \lambda & \lambda ^2 & \ldots & \lambda ^m & z & * & \ldots
      & * \end{array} \right).
  \]
  As $p>3$, we can choose $y,z \in \mathbb{Z}/p\mathbb{Z}$ with $y
  \not \equiv z$ and $y,z \not \equiv \lambda ^{m+1} \pmod{p}$ so that
  the condition in (3)(a) is satisfied.

  Finally suppose that $\mathbf{a} = ()$.  Then
  \[
  M(\widetilde{\mathbf{E}}) = \left( \begin{array}{cccccc}
      1 & y & b_1+yc_1 & * & \ldots & * \\
      1 & z & b_1+zc_1 & * & \ldots & *\end{array} \right)
  \]
  for suitable $b_1,c_1 \in \mathbb{Z}/p\mathbb{Z}$.  We can choose
  $y,z \in \mathbb{Z}/p\mathbb{Z}$ with $y\not \equiv z$ such that
  \[
  y^2 \not \equiv b_1 +yc_1 \quad \text{and} \quad z^2 \not \equiv b_1
  +zc_1 \pmod{p},
  \]
  because quadratic equations have at most two solutions and $p>3$.
  Once more, the condition in (3)(a) is fulfilled.
\end{proof}

The next result mimics~\cite[Lemma 3.2]{Amaia}, which applies to
GGS-groups.  We remark that there are no new exceptions, in addition
to the GGS-group
\begin{equation}\label{eq:exception}
  \mathcal{G}= \langle a,b \rangle \quad \text{with} \quad \psi_1(b)
  = (a,a,\ldots,a,b),
\end{equation}
arising from a constant defining vector~$(1,\ldots,1)$.

\begin{proposition}
  \label{Amaia-gamma}
  Let $G = \langle a, b_1, \ldots, b_r \rangle$ be a multi-edge spinal
  group that is not conjugate to $\mathcal{G}$ in $\Aut(T)$.  Then
  \[
  \psi_1(\gamma_3(\stab_G(1)))=\gamma_3(G)\times \overset{p}{\ldots}
  \times \gamma_3(G).
  \]
  In particular,
  \[
  \gamma_3(G) \times \overset{p}{\ldots} \times \gamma_3(G) \subseteq
  \psi_1(\gamma_3(G)).
  \]
\end{proposition}

\begin{proof}
  From $\psi_1(\stab_G(1))\subseteq G\times \overset{p}{\ldots} \times
  G$, we deduce that $\psi_1(\gamma_3(\stab_G(1)))\subseteq
  \gamma_3(G)\times \overset{p}{\ldots} \times \gamma_3(G)$, and so it
  suffices to prove the reverse inclusion.  For $r=1$, the result has
  been proved in~\cite[Lemma~3.2]{Amaia}.  Hence suppose that $r\ge
  2$. For convenience, we prove the result in the isomorphic setting
  given by Lemma~\ref{linearalgebra}. That is, since
  \[
  \psi_1(\gamma_3(\stab_G(1)))^f =
  \psi_1(\gamma_3(\stab_G(1))^f)=\psi_1(\gamma_3(\stab_{G^f}(1)))
  \]
  and
  \[
  (\gamma_3(G)\times \ldots \times \gamma_3(G))^f = \gamma_3(G^f)
  \times \ldots \times \gamma_3(G^f), 
  \]
  for $f$ as in Lemma~\ref{linearalgebra}, we may assume that
  $\psi_1(b_i)=(a^{e_{i,1}},\ldots,a^{e_{i,p-1}},b_i)$ with
  $e_{i,1}=1$, for $i \in \{1,\ldots,r\}$, and that the additional
  assertions (2), (3), (4) of Lemma~\ref{linearalgebra} hold.
  Moreover, by Proposition~\ref{transitive} the group $G$ acts
  spherically transitively, hence it suffices to show that
  \[
  \gamma_3(G) \times 1 \times \ldots \times 1 \subseteq
  \psi_1(\gamma_3(\stab_G(1)).
  \]
  We divide the argument into two cases.

  \smallskip

  \noindent \underline{Case 1}: $(r,p) \ne (2,3)$.  By
  Proposition~\ref{transitive}, the group $G$ is fractal.  Since
  $\gamma_3(G)=\langle [a,b_i,a],[a,b_i,b_j] \mid 1 \le i,j \le r
  \rangle^G$, it suffices to construct elements in
  $\gamma_3(\stab_G(1))$ whose images under $\psi_1$ yield
  \[
  ([a,b_i,a],1,\ldots,1) \quad \text{and} \quad
  ([a,b_i,b_j],1,\ldots,1), \qquad 1 \le i,j \le r.
  \] 

  First suppose that $i\in \{1,\ldots,r\}$ with $e_{i,p-1} \equiv 0
  \pmod{p}$.  Then
  \[
  \psi_1(b_i)=(a,a^{e_{i,2}},\ldots,a^{e_{i,p-2}},1,b_i).
  \]
  Noting that $[b_i,b_j^a]=(*,1,\ldots,1,*)$ for $j \in
  \{1,\ldots,r\}$, where the symbols $*$ denote unspecified entries and the second $*$ equals 0 if $j=i$,
  we deduce that
  \begin{align*}
    \psi_1([b_i,b_i^a,b_i]) & =([a,b_i,a],1,\ldots,1), \\
    \psi_1([b_i,b_i^a,b_i^a]) & =([a,b_i,b_i],1,\ldots,1),
  \end{align*}
  and for $j \in \{1,\ldots,r\}$ with $j\ne i$,
  \begin{align*}
    \psi_1([b_i,b_i^a,b_j^a]) & =([a,b_i,b_j],1,\ldots,1), \\
    \psi_1([b_i,b_j^a,b_i^a]) & =([a,b_j,b_i],1,\ldots,1).
  \end{align*}

  Next suppose that $i \in \{1,\ldots,r\}$ with $e_{i,p-1}\not \equiv
  0 \pmod{p}$. By properties (3) and (4) in Lemma~\ref{linearalgebra},
  there exists $k\in \{2,\ldots,p-2\}$ such that $e_{i,k-1} e_{i,k+1}
  \not \equiv e_{i,k}^2 \pmod{p}$, apart from an exceptional case
  which only occurs for $r=2$ and which we deal with separately below.
  Set
  \[
  g_{i,k}=(b_i^{a^{p-k+1}})^{e_{i,k}} (b_i^{a^{p-k}})^{-e_{i,k-1}}
  \]
  so that
  \[
  \psi_1(g_{i,k})=(a^{e^2_{i,k} - e_{i,k-1}e_{i,k+1}},*,\ldots,*,1).
  \]
  Since $e^2_{i,k} - e_{i,k-1} e_{i,k+1} \not \equiv 0 \pmod{p}$,
  there is a power $g_i$ of $g_{i,k}$ such that
  \[
  \psi_1(g_i)=(a,*,\ldots,*,1).
  \]
  Additionally, since
  \[
  \psi_1(b_i^a (b_i^{a^{p-1}})^{-e_{i,p-1}}) =(b_i a^{-e_{i,2}
    e_{i,p-1}},*,\ldots,*, 1),
  \]
  with the help of $g_i$ we get an element $h_i \in \stab_G(1)$ such
  that
  \[
  \psi_1(h_i)=(b_i,*,\ldots,*,1).  
  \]
  Consequently, we obtain
  \begin{align*}
    \psi_1([b_i,b_i^a,g_i]) &=([a,b_i,a],1,\ldots,1) \\
    \psi_1([b_i,b_i^a,h_i]) &=([a,b_i,b_i],1,\ldots,1),
  \end{align*}
  and for $j \in \{1,\ldots,r\}$ with $j\ne i$,
  \[
  \psi_1([b_i,b_j^a,h_i]) =([a,b_j,b_i],1,\ldots,1)
  \]
  and
  \begin{align*}
    \psi_1([b_i,b_i^a,h_j]) & =([a,b_i,b_j],1,\ldots,1) & & \text{if
      $e_{j,p-1} \not \equiv 0$,} \\
    \psi_1([b_i,b_i^a,b_j^a])& =([a,b_i,b_j],1,\ldots,1) & & \text{if
      $e_{j,p-1} \equiv 0$ (as in the previous part)}.
  \end{align*}
  Thus we have constructed all necessary elements.

  It remains to deal with the exceptional case which occurs only for
  $r=2$, and hence $p>3$.  According to property (3b) in
  Lemma~\ref{linearalgebra} we have
  \[
  e_1 =(1, 0, \ldots, 0), \quad e_2=(1, 0 , \ldots , 0 ,1),
  \]
  so that
  \[
  b_1=(a,1,\ldots,1,1,b_1), \quad b_2=(a,1,\ldots,1,a,b_2).
  \]
  We simply replace $g_2$ in the above argument by $b_2^{a^2} =
  (a,b_2,a,1,\ldots,1)$, and then proceed similarly.

  \smallskip

  \noindent \underline{Case 2}: $(r,p) = (2,3)$.  Here $G'=\gamma_2(G)=\langle
  [a,b_1], [a,b_2]\rangle ^G$ and as $\langle ab_2, b_2a^{-1}\rangle =
  \langle b_2^2 ,b_2 a^{-1}\rangle = \langle a,b_2 \rangle$, we have
  \[
  \gamma_3(G)=\langle [a,b_1,a], [a,b_1,b_1], [a,b_1,b_2],
  [a,b_2,b_1], [a,b_2,ab_2], [a,b_2,b_2 a^{-1}] \rangle ^G.
  \]
  By property (2) in Lemma~\ref{linearalgebra}, we may assume
  $\mathbf{e}_1=(1,0)$ and $\mathbf{e}_2=(1,-1)$ so that
  \[
  \psi_1(b_1) =(a,1,b_1) \quad \text{and} \quad \psi_1(b_2)
  =(a,a^{-1},b_2).
  \]
  Arguing as in the previous case, it suffices to manufacture, in
  addition to the elements already constructed there, elements of
  $\gamma_3(\stab_G(1))$ whose images under $\psi_1$ are
  $([a,b_2,ab_2],1,1)$ and $([a,b_2,b_2 a^{-1}],1,1)$.  We compute
  \[
  \psi_1(b_2 b_2^{a}) =(ab_2,1,b_2a^{-1}) \quad \text{and} \quad
  \psi_1(b_2^a b_2^{a^2}) =(b_2 a^{-1},ab_2,1),
  \]
  yielding
  \begin{align*}
    \psi_1([b_2,b_2^a,b_2^a b_2^{a^2}]) &= ([a,b_2,b_2a^{-1}],1,1), \\
    \psi_1([b_2^{-a^2},b_2^a,b_2 b_2^a]) &= ([a,b_2,ab_2],1,1).
  \end{align*}
\end{proof} 

The following consequence paves the way to proving branchness.

\begin{proposition} \label{g3inR} Let $G$ be a multi-edge spinal group
  that is not $\Aut(T)$-conjugate to the GGS-group $\mathcal{G}$ in
  \eqref{eq:exception}. Then $\gamma_3(G) \subseteq \rstab_G(u)_u$ for
  every vertex $u$ of~$T$, after the natural identification of
  subtrees.
\end{proposition}

\begin{proof}
  Let $u$ be a vertex of $T$ at level~$n$. We induct on $n$. If $n=0$
  then $u=\varnothing$ is the root vertex and the claim holds
  trivially.  Now suppose that $n>0$. Writing $u$ as $u=vy$, where $v$
  is a vertex at level $n-1$ and $y \in X$, we conclude by
  induction that $\gamma_3(G) \subseteq \rstab_G(v)_v$. By
  Proposition~\ref{Amaia-gamma},
  \[
  \gamma_3(G) \times \overset{p}{\ldots} \times \gamma_3(G) \subseteq
  \psi_1(\gamma_3(G))\subseteq \psi_1(\rstab_G(v)_v)
  \]
  so that, in particular,
  \[
  1 \times \ldots \times 1 \times \gamma_3(G) \times 1 \times \ldots
  \times 1 \subseteq \psi_1 (\rstab_G(v)_v),
  \]
  where $\gamma_3(G)$ is located at position $u$ in the subtree $T_v$
  rooted at~$v$. Hence $\gamma_3(G)\subseteq \rstab_G(u)_u$.
\end{proof}

\begin{proposition} \label{prop:branch} Let $G$ be a multi-edge spinal
  group that is not $\Aut(T)$-conjugate to the GGS-group $\mathcal{G}$
  in~\eqref{eq:exception}. Then $G$ is a branch group.
\end{proposition}

\begin{proof}
  In view of Proposition~\ref{transitive} it suffices to show that
  every rigid level stabilizer $\rstab_G(n)$ is of finite index in
  $G$. The nilpotent quotient $G/\gamma_3(G)$ is generated by finitely
  many elements of finite order and hence finite.  Thus $\gamma_3(G)$
  has finite index in $G$. By Proposition~\ref{g3inR}, the image of
  $\rstab_G(n)$ under the maps $\psi_n$ contains the direct product of
  $p^n$ copies of $\gamma_3(G)$. Since the image of any level
  stabilizer $\stab_G(n)$ under the injective map $\psi_n$ is
  contained in the direct product of $p^n$ copies of $G$, we deduce
  that $\rstab_G(n)$ is of finite index in $\stab_G(n)$ and hence in
  $G$.
\end{proof}

\begin{theorem}[{\cite[Theorem 4]{NewHorizons}}]
  \label{justinfinite}
  A branch group $G$ with branch structure
  $((L_n)_{n=1}^\infty,(H_n)_{n=1}^\infty)$ is just infinite if and
  only if for each $n\ge 1$, the index of the commutator subgroup
  $L_n'$ in $L_n$ is finite.
\end{theorem}

\begin{corollary}[{\cite[Section~7]{NewHorizons}}] \label{justinf}
  Every finitely generated, torsion branch group $G$ is just infinite.
\end{corollary}

\begin{proof}
  As $G$ is branch, $L_n=L_n^G$ is of finite index in $G$. Hence $L_n$
  is a finitely generated torsion group and $L_n/L_n'$ is finite
  abelian.
\end{proof}

We do not know a proof that the GGS-group $\mathcal{G}$ in
\eqref{eq:exception} is not branch.  From properties that were
established in~\cite{Amaia} we derive the following result.

\begin{proposition} \label{prop:special-group}
  The GGS-group $\mathcal{G}$ in~\eqref{eq:exception} is not just
  infinite.
\end{proposition}

\begin{proof}
  Write $G = \mathcal{G} = \langle a,b \rangle$ with $\psi_1(b) =
  (a,\ldots,a,b)$, and put $K = \langle ba^{-1} \rangle^G$.
  From~\cite[Section~4]{Amaia} we have that
  \begin{enumerate} 
  \item $|G:K|=p$ and $K' = \langle [(ba^{-1})^a,ba^{-1}]
    \rangle^G \le \stab_G (1)$;
  \item $ |G / K'\stab_G(n)| = p^{n+1}$ for every $n \in \mathbb{N}$
    with $n\ge 2$.
  \end{enumerate}
  Hence $K'$ is a non-trivial normal subgroup of infinite index in
  $G$.
\end{proof}

What about just infinite multi-edge spinal groups that are not
torsion? For $p\ge 5$, it is shown in \cite[Example~7.1]{NewHorizons}
that the non-torsion group $G=\langle a, b \rangle$ with $\psi_1
(b)=(a,1,\ldots,1,b)$ is just infinite, and more generally in
\cite[Example~10.2]{NewHorizons} that $G=\langle a,b \rangle$ with
$\psi_1(b)=(a^{e_1},a^{e_2},\ldots,a^{e_{p-4}},1,1,1,b)$ where $e_1
\not \equiv 0$ is just infinite. For the latter example, when
$\sum_{i=1}^{p-4} e_i \not \equiv 0$ (mod $p$), then the group is
non-torsion.

Now let $G$ be the multi-edge spinal group with defining vectors
$\mathbf{e}_i$ of the form
$(e_{i,1},e_{i,2},\ldots,e_{i,p-2},e_{i,p-1})$ satisfying $e_{i,1}
\not \equiv 0$ (mod $p$) and $e_{i,p-3} \equiv e_{i,p-2} \equiv
e_{i,p-1} \equiv 0$ (mod $p$) for every $i\in \{1,\ldots,r\}$.  As in
\cite[Example~10.2]{NewHorizons}, it can be shown that $G$ is just
infinite, and furthermore when $\sum_{j=1}^{p-4} e_{i,j} \not \equiv
0$ (mod $p$) for at least one $i\in \{1,\ldots, r\}$, then $G$ is
non-torsion. It is not always the case that the last three entries of
the defining vectors are to be zero. For example, the non-torsion
multi-edge spinal group $G$ with $e_{i,1} \equiv e_{i,p-2} \equiv
e_{i,p-1} \equiv 0$ (mod $p$) and $e_{i,2} \not \equiv 0$ (mod $p$) is
likewise just infinite.


\section{Theta maps} 
Here we determine the abelianization $G/G'$ of a multi-edge spinal
group~$G$. Then we define a natural length function on elements of the
commutator subgroup $G'$. Akin to Pervova's work \cite{Pervova4}, we
introduce two theta maps $\Theta_1,\Theta_2 \colon G' \rightarrow G'$
which are key to establishing that all maximal subgroups of $G$ are of
finite index. We prove that the length of every element of the
commutator subgroup of length at least 3 decreases under repeated
applications of a combination of these maps. Our use of \textit{two}
theta maps, instead of one as in \cite{Pervova4} allows us to
significantly simplify the calculations.

\subsection{Abelianization of multi-edge spinal groups}
Recall that every element $g$ of the free product $\Fr_{\lambda \in
  \Lambda} \Gamma_{\lambda}$ of a family of groups
$(\Gamma_\lambda)_{\lambda \in \Lambda}$ can be uniquely represented
as a \textit{reduced} word in $\sqcup_{\lambda \in
  \Lambda}\Gamma_{\lambda}$, i.e., a word $g = g_1 g_2 \cdots g_n$,
where $n\in \mathbb{N} \cup \{0\}$, $\lambda_1,\ldots,\lambda_n \in
\Lambda$ with $\lambda_i \ne \lambda_{i+1}$ for $1 \le i \le n-1$, and
$1 \ne g_i \in \Gamma_{\lambda_i}$ for each $i \in \{1,\ldots,n\}$.

Let $G = G_{\mathbf{E}} =\langle a,b_1,\ldots,b_r\rangle$ be a
multi-edge spinal group acting on the regular $p$-adic rooted tree
$T$, for an odd prime $p$.  Here $\mathbf{E}$ is the $r$-tuple of
defining vectors $\mathbf{e}_i = (e_{i,1},\ldots,e_{i,p-1})$, for
$i\in \{1,\ldots,r\}$.

In order to study $G/G'$ we consider
\begin{multline} \label{H} H = \langle \hat{a}, \hat{b}_1, \ldots,
  \hat{b}_r \mid \\ \hat{a}^p=\hat{b}_1^p=\ldots =\hat{b}_r^p=1,
  \text{ and } [\hat{b}_i,\hat{b}_j]=1\text{ for }1\le i,j \le
  r\rangle,
\end{multline}
the free product $\langle \hat{a}\rangle * \langle
\hat{b}_1,\ldots,\hat{b}_r \rangle$ of a cyclic group $\langle
\hat{a}\rangle \cong C_p$ and an elementary abelian group $\langle
\hat{b}_1,\ldots,\hat{b}_r \rangle \cong C_p^r$.  There is a unique
epimorphism $\pi\colon H \rightarrow G$ such that $\hat{a} \mapsto a$
and $\hat{b}_i \mapsto b_i$ for $i\in \{1,\ldots,r\}$, inducing an
epimorphism from $H/H'\cong \langle \hat a \rangle \times \langle
\hat{b}_1,\ldots,\hat{b}_r \rangle \cong C_p^{r+1}$ onto $G/G'$.  We
want to show that the latter is an isomorphism; see
Proposition~\ref{abelianization} below.

Let $h \in H$.  As discussed, each $h$ can be uniquely represented in
the form
\begin{equation} \label{eq:4.1} h = \hat{a}^{s_1} \cdot
  (\hat{b}_1^{\beta_{1,1}} \cdots \hat{b}_r^{\beta_{r,1}}) \cdot
  \hat{a}^{s_2} \cdot \ldots \cdot \hat{a}^{s_m} \cdot
  (\hat{b}_1^{\beta_{1,m}} \cdots \hat{b}_r^{\beta_{r,m}}) \cdot
  \hat{a}^{s_{m+1}},
\end{equation}
where $m\in \mathbb{N}\cup \{0\}$ and $s_1, \ldots, s_{m+1},
\beta_{1,1}, \ldots, \beta_{r,m} \in \mathbb{Z}/p\mathbb{Z}$ with
\begin{align*}
  s_i & \not \equiv 0 \pmod{p} \qquad \text{for $i\in
    \{2,\ldots,m\}$,} \\
\intertext{and  for each $j \in \{1,\ldots,m\}$,}
\beta_{i,j} & \not \equiv 0 \pmod{p} \qquad \text{for at least one $i\in
  \{1,\ldots,r\}$.}
\end{align*}
We denote by $\partial (h) = m$ the \textit{length} of $h$, with
respect to the factor $\langle \hat{b}_1, \ldots, \hat{b}_r \rangle$.  Clearly,
for $h_1,h_2 \in H$ we have
\begin{equation} \label{eq:product-h1-h2}
  \partial(h_1 h_2) \leq \partial(h_1) + \partial(h_2).
\end{equation}
In addition, we define \textit{exponent maps}
\begin{equation} \label{eq:def-eps-hat}
  \begin{split}
    \epsilon_{\hat a}(h) & =
    \sum\nolimits_{j=1}^{m+1} s_j \in \mathbb{Z}/p\mathbb{Z} \quad \text{and} \\
    \epsilon_{\hat{b}_i}(h) & = \sum\nolimits_{j=1}^m \beta_{i,j} \in
    \mathbb{Z}/p\mathbb{Z} \quad \text{for $i \in \{1,\ldots,r\}$}
  \end{split}
\end{equation}
with respect to the generating set $\hat{a},
\hat{b}_1,\ldots,\hat{b}_r$.

The surjective homomorphism
\begin{equation} \label{eq:abelian-H} H \rightarrow
  (\mathbb{Z}/p\mathbb{Z}) \times (\mathbb{Z}/p\mathbb{Z})^r, \quad h
  \mapsto (\epsilon_{\hat a}(h), \epsilon_{\hat b_1}(h), \ldots,
  \epsilon_{\hat b_r}(h))
\end{equation}
has kernel $H'$ and provides an explicit model for the
abelianization~$H/H'$. The group $L(H) = \langle
\hat{b}_1,\ldots,\hat{b}_r\rangle^H$ is the kernel of the surjective
homomorphism
\[
H \rightarrow \mathbb{Z}/p\mathbb{Z}, \quad h \mapsto \epsilon_{\hat a}(h).
\]
Each element $h \in L(H)$ can be uniquely represented by a word of the
form
\begin{equation} \label{eq:4.2} h= (\hat{c}_1)^{\hat{a}^{t_1}} \cdots
  (\hat{c}_m)^{\hat{a}^{t_m}},
\end{equation}
where $m\in \mathbb{N}\cup\{0\}$ and $t_1, \ldots, t_m \in
\mathbb{Z}/p\mathbb{Z}$ with $t_j\not \equiv t_{j+1} \pmod{p}$ for
$j\in \{1,\ldots,m-1\}$, and for each $j\in \{1,\ldots,m\}$,
\begin{equation} \label{eq:4.2-add-on}
\hat{c}_j = \hat{b}_1^{\beta_{1,j}} \cdots \hat{b}_r^{\beta_{r,j}} \in
\langle \hat{b}_1, \ldots, \hat{b}_r \rangle \setminus \{1\}.
\end{equation}

Let $\alpha$ denote the cyclic permutation of the factors of $H\times
\overset{p}{\ldots} \times H$ corresponding to the $p$-cycle $(1 \, 2
\, \ldots \, p)$.  We consider the homomorphism
\[
\Phi \colon L(H) \rightarrow H\times \overset{p}{\ldots} \times H
\]
defined by 
\[
\Phi(\hat{b}_i^{\hat{a}^k}) =
(\hat{a}^{e_{i,1}},\ldots,\hat{a}^{e_{i,p-1}},\hat{b}_i)^{\alpha^k}
\quad \text{for $i \in \{1,\ldots,r\}$, $k \in
  \mathbb{Z}/p\mathbb{Z}$.}
\]

\begin{lemma} \label{ceil} Let $H$ be as above, and $h \in L(H)$ with
  $\Phi(h)=(h_1,\ldots,h_p)$.  Then $\sum_{i=1}^p \partial(h_i)
  \le \partial(h)$, and $\partial(h_i) \le \lceil
  \frac{\partial(h)}{2}\rceil$ for each $i \in \{1,\ldots,p\}$.
\end{lemma}

\begin{proof}
  Suppose that $h$ is of length $\partial(h) = m$ as
  in~\eqref{eq:4.2}.  For each $j \in \{1,\ldots,m\}$ the factor
  $(\hat{c}_j)^{\hat{a}^{t_j}}$ in~\eqref{eq:4.2} contributes to
  precisely one coordinate of $\Phi(h)$ a factor $\hat{c}_j$ and to
  all other coordinates a power of $\hat{a}$. Therefore
  $\sum_{i=1}^p \partial(h_i) \le m$.

  Now let $i \in \{1,\ldots,p\}$.  The maximum length in the $i$th
  coordinate occurs when $h_i$ is of the form $\hat{c}_* \hat{a}^*
  \hat{c}_* \ldots \hat{a}^* \hat{c}_*$ with $m$ factors, where the
  symbols $*$ represent suitable indices or exponents.  Therefore
  $\partial(h_i) \le \lceil m/2\rceil$.
\end{proof}

The following proposition provides a recursive presentation for a
multi-edge spinal group.  It can be extracted from a result of
Rozhkov~\cite{Rozhkov}; a self-contained proof for multi-edge spinal
groups is included in~\cite{thesis}.

\begin{proposition} \label{K} Let $G=\langle a,b_1,\ldots,b_r\rangle$
  be a multi-edge spinal group, and $H$ as in~\eqref{H}.  Consider the
  subgroup $K=\bigcup_{n=0}^{\infty} K_n$ of $H$, where
  \[
  K_0=\{1\} \quad \text{and} \quad K_n = \Phi^{-1}(K_{n-1} \times
  \ldots \times K_{n-1}) \text{ for $n\ge 1$.}
  \]
  Then $K \subseteq L(H) = \langle \hat{b}_1,\ldots, \hat{b}_r \rangle
  ^H$, and $K$ is normal in $H$.  Moreover, the epimorphism $\pi\colon
  H \rightarrow G$ given by $\hat{a} \mapsto a$, $\hat{b}_i \mapsto
  b_i$, for $i \in \{1,\ldots,r\}$, has $\ker(\pi)=K$.  In particular,
  $G \cong H/K$.
\end{proposition}

Next we describe the abelianization of a multi-edge spinal
group.

\begin{proposition}
  \label{abelianization}
  Let $G=\langle a,b_1,\ldots,b_r\rangle$ be a multi-edge spinal
  group, and $H$ as in~\eqref{H}.  Then the map $H \rightarrow
  (\mathbb{Z}/p\mathbb{Z}) \times (\mathbb{Z}/p\mathbb{Z})^r$ in
  \eqref{eq:abelian-H} factors through $G/G'$.  Consequently,
  \[
  G/G' \cong H/H' \cong C_p^{r+1}.
  \]
\end{proposition}

\begin{proof}
  Below we prove that
  \begin{equation} \label{eq:4.4} \Phi^{-1}(H'\times
    \overset{p}{\ldots} \times H') \le H'.
  \end{equation}
  Let $K = \bigcup_{n=0}^{\infty} K_n \le L(H)$ be as in
  Proposition~\ref{K} so that the natural epimorphism $\pi \colon H
  \rightarrow G$ has $\ker(\pi)= K$, and $G\cong H/K$.
  From~\eqref{eq:4.4}, we deduce by induction that $K_n \le H'$ for
  all $n \in \mathbb{N} \cup \{0\}$, hence $K\le H'$ and $G/G' \cong
  H/H'K = H/H'$.

  \smallskip

  It remains to justify~\eqref{eq:4.4}.  Consider an arbitrary element
  $h\in L(H)$ as in \eqref{eq:4.2} and \eqref{eq:4.2-add-on}.  We
  write $\Phi(h)=(h_1,\ldots,h_p)$.  For $i\in \{1,\ldots,r\}$ and $k
  \in \{1,\ldots,p\}$, let $\epsilon_{\hat{b}_{i},k}(h)$ be the sum of
  exponents $\beta_{i,j}$, $j \in \{1,\ldots,m\}$ with $t_j = k$, so
  that $\epsilon_{\hat{b}_i}(h_k) = \epsilon_{\hat{b}_i,k}(h)$.  It
  follows that for each $i\in \{1,\ldots,r\}$,
  \begin{equation}\label{eq:4.3}
    \epsilon_{\hat{b}_i}(h) = \sum\nolimits_{j =1}^m \beta_{i,j} =
    \sum\nolimits_{k=1}^p 
    \epsilon_{\hat{b}_i,k}(h) = \sum\nolimits_{k=1}^p \epsilon_{\hat{b}_i}(h_k).
  \end{equation}

  Now suppose that $h\not \in H'$. From \eqref{eq:abelian-H} and
  $\epsilon_{\hat a}(H) =0$ we deduce that $\epsilon_{\hat{b}_i}(h)
  \not \equiv 0$ for at least one $i\in \{1,\ldots,r\}$.  Thus
  \eqref{eq:4.3} implies that $\epsilon_{\hat{b}_i}(h_k) \not \equiv
  0$ for some $k \in \{1,\ldots,p\}$ and $\Phi(h)\not \in H' \times
  \overset{p}{\ldots} \times H'$. Therefore \eqref{eq:4.4} holds.
\end{proof}

As above, let $G=\langle a,b_1,\ldots,b_r \rangle$ be a multi-edge
spinal group, and $\pi \colon H \rightarrow G$ the natural epimorphism
with $H$ as in~\eqref{H}.  The \textit{length} of $g \in G$ is
\[
\partial(g) = \min \{ \partial(h) \mid h \in \pi^{-1}(g) \}.
\]
Based on \eqref{eq:product-h1-h2}, one easily shows that for $g_1,g_2
\in G$,
\begin{equation} \label{eq:product-g1-g2}
  \partial(g_1 g_2) \leq \partial(g_1) + \partial(g_2).
\end{equation}
Moreover, using Proposition~\ref{abelianization} we may define
$\epsilon_a(g), \epsilon_{b_1}(g),\ldots,\epsilon_{b_r}(g) \in
\mathbb{Z}/p\mathbb{Z}$ via any pre-image $h \in \pi^{-1}(g)$:
\begin{equation} \label{eq:def-eps}
  (\epsilon_a(g),\epsilon_{b_1}(g),\ldots, \epsilon_{b_r}(g)) =
  (\epsilon_{\hat a}(h),\epsilon_{\hat b_1}(h),\ldots, \epsilon_{\hat
    b_r}(h)).
\end{equation}

We record the following direct consequences of Lemma~\ref{ceil}.

\begin{lemma} \label{shortening} Let $G$ be a multi-edge spinal group
  as above, and $g \in \stab_G(1)$ with $\psi_1(g)=(g_1,\ldots,g_p)$.
  Then $\sum_{i=1}^p \partial(g_i) \le \partial(g)$, and
  $\partial(g_i) \le \lceil \frac{\partial(g)}{2}\rceil$ for each
  $i\in \{1,\ldots,p\}$.

  In particular, if $\partial(g)>1$ then $\partial(g_i) < \partial(g)$
  for every $i \in \{1,\ldots,p\}$.
\end{lemma}

\subsection{Length reduction} \label{sec:theta-def} We continue to
consider a multi-edge spinal group $G = G_{\mathbf{E}} =\langle
a,b_1,\ldots,b_r\rangle$ acting on the regular $p$-adic rooted tree
$T$, for an odd prime $p$.  Here $\mathbf{E}$ is the $r$-tuple of
defining vectors $\mathbf{e}_i = (e_{i,1},\ldots,e_{i,p-1})$, for
$i\in \{1,\ldots,r\}$.  In preparation for the investigation of
maximal subgroups of $G$, we introduce in the present section two
length decreasing maps $\Theta_1, \Theta_2 \colon G' \rightarrow G'$.
Based on Lemma~\ref{dagger}, we assume that $e_{1,1}=1$.  Also, let
\[
n = \max \big\{ j \in \{1,\ldots,p-1\} \mid e_{1,j}\not \equiv 0 \pmod{p}
\big\}.
\]
Generically, we have $n\ge 2$, while the exceptional case $n=1$
corresponds to $b_1$ of the form $\psi_1(b_1) = (a,1,\ldots,1,b_1)$.
The special case $n=1$ will be dealt with slightly differently in what
follows, and it only applies to a specific family of groups.  We
remark that, if $G$ is torsion, then Theorem~\ref{torsion}
automatically yields $n \geq 2$.

Clearly, $G'=\langle [a,b_i] \mid i\in \{1,\ldots,r\}\rangle^G$ is a
subgroup of $\stab_G(1)=\langle b_1,\ldots,b_r\rangle ^G$.  Every
$g\in \stab_G(1)$ has a decomposition
\[
\psi_1(g)= (g_1,\ldots,g_p),
\]
where each $g_j \in U_{u_j}^G \cong G$ is an element of the upper
companion group acting on the subtree rooted at a first level vertex
$u_j$, $j \in \{1,\ldots,p\}$, and we define
\begin{equation} \label{eq:def-phi} \varphi_j\colon \stab_G(1)
  \rightarrow \Aut(T_{u_j}), \quad \varphi_j(g)=g_j.
\end{equation}
It is customary and useful to write $(g_1,\ldots,g_p)$ in place of $g
\in \stab_G(1)$ to carry out certain computations.

We are interested in projecting, via $\varphi_p$, the first level
stabilizer $\stab_M(1)$ of a subgroup $M \leq G$, containing $b_1$ and
an `approximation' $az \in a G'$ of $a$, to a subgroup of
$\Aut(T_{u_p})$.  Writing $\psi_1(z) =(z_1,\ldots,z_p)$ and
conjugating $b_1$ by $(az)^{-1}$, we obtain
\begin{align*}
  b_1^{(az)^{-1}} & = (z\cdot
  (a,a^{e_{1,2}},\ldots,a^{e_{1,p-1}},b_1)\cdot z^{-1})^{a^{-1}} \\
  & = (a^{z_1^{-1}},(a^{e_{1,2}})^{z_2^{-1}}, \ldots,
  (a^{e_{1,p-1}})^{z_{p-1}^{-1}}, b_1^{z_p^{-1}})^{a^{-1}} \\
  &
  =((a^{e_{1,2}})^{z_2^{-1}}, \ldots, (a^{e_{1,p-1}})^{z_{p-1}^{-1}},
  b_1^{z_p^{-1}}, a^{z_1^{-1}}).
\end{align*}
Therefore
\[
\varphi_p(b_1^{(az)^{-1}}) = a^{z_1^{-1}} = a[a,z_1^{-1}]
\]
and this motivates us to define
\[
\Theta_1\colon G' \rightarrow G', \quad \Theta_1(z)=[a,z_1^{-1}].
\]

The map $\Theta_2$ is obtained similarly. As $e_{1,n} \not \equiv 0$,
we find $k \in \mathbb{Z}/p\mathbb{Z}$ such that $k e_{1,n} \equiv 1
\pmod{p}$.  Writing $\widetilde{e}_{1,j} = k e_{1,j}$ for $j \in
\{1,\ldots,p-1\}$, we obtain by induction on $p-n$ that
\begin{align*}
  \big(b_1^k\big)^{(az)^{p-n}} & = \big( z^{-1} \cdot \big(
  a^{\widetilde{e}_{1,1}}, \ldots, a^{\widetilde{e}_{1,n-1}}, a,
  a^{\widetilde{e}_{1,n+1}},
  \ldots, a^{\widetilde{e}_{1,p-1}}, b_1^k \big)^a \cdot z \big)^{(az)^{p-n-1}} \\
  & = \big( \big( b_1^k \big)^{z_1}, \big(a^{\widetilde{e}_{1,1}}
  \big)^{z_2}, \ldots, (a^{\widetilde{e}_{1,n-1}})^{z_n}, a^{z_{n+1}},
  *, \ldots, *)^{(az)^{p-n-1}} \\
  & =(*,\ldots,*,a^{z_{n+1} z_{n+2} \cdots z_p}),
\end{align*}
where the symbols $*$ represent unspecified components.  Therefore
\[
\varphi_p \big( \big(b_1^k \big)^{(az)^{p-n}} \big) = a^{z_{n+1}
  \cdots z_p} = a [a,z_{n+1} \cdots z_p]
\]
and this motivates us to define
\[
\Theta_2\colon G'\rightarrow G', \quad
\Theta_2(z)=[a,z_{n+1} \ldots z_p].
\]

To deal with the case $n=1$, we define $\boldsymbol{\mathcal{E}}$ to
be the family of all multi-edge spinal groups $G=\langle
a,b_1,\ldots,b_r \rangle$ that satisfy
\begin{equation} \label{eq:H}
  \begin{cases}
    \psi_1(b_1) =(a,1,\ldots,1,b_1) & \\
    e_{i,1}\equiv 1 \pmod{p} & \text{for every $i \in \{1,\ldots,r\}$} \\
    e_{i,p-1} \not \equiv 0 \pmod{p} & \text{for at least one $i \in
    \{1,\ldots,r\}$.}
  \end{cases}
\end{equation}
We remark that, by Theorem~\ref{torsion}, there are no torsion groups
in $\boldsymbol{\mathcal{E}}$.

\begin{theorem} \label{theta} Let $G=\langle a,b_1,\ldots,b_r \rangle$
  be a multi-edge spinal group acting on the regular $p$-adic rooted
  tree $T$, for an odd prime $p$. Suppose $G$ is not
  $\Aut(T)$-conjugate to a group in $\boldsymbol{\mathcal{E}}$. Then
  the length $\partial (z)$ of an element $z \in G'$ decreases under
  repeated applications of a suitable combination of the maps
  $\Theta_1$ and $\Theta_2$ down to length $0$ or~$2$.
\end{theorem}

\begin{proof}
  Let $z \in G'$.  We observe that $\partial(z) \neq 1$; see
  Proposition~\ref{abelianization}.  Suppose that $\partial (z) = m
  \ge 3$.  Then $z\in G'\subseteq \stab_G(1)$ has a decomposition
  \[
  \psi_1(z) = (z_1,\ldots,z_p).
  \]
  From Lemma~\ref{shortening} and \eqref{eq:product-g1-g2} we obtain
  $\partial(z_j) \le \lceil \frac{m}{2} \rceil$ for $j \in
  \{1,\ldots,p\}$ and
  \[
  \partial(z_1) + \partial(z_{n+1}\cdots z_p) \le m.
  \]
  If $\partial (z_1)< \frac{m}{2}$ then $\partial (\Theta_1 (z))<m$,
  and likewise if $\partial (z_{n+1}\cdots z_p) < \frac{m}{2}$ then
  $\partial(\Theta_2(z))<m$.  Hence we may suppose that $m = 2\mu$ is
  even and
  \[
  \partial(z_1) = \partial(z_{n+1} \cdots z_p) = \mu.
  \]

  We write $z_{n+1} \cdots z_p$ as
  \[
  a^{s_1} \cdot c_1 \cdot a^{s_2} \cdot \ldots \cdot a^{s_\mu} \cdot
  c_\mu \cdot a^{s_{\mu+1}},
  \]
  where $s_1, \ldots, s_{\mu+1} \in \mathbb{Z}/p\mathbb{Z}$ with $s_i
  \not \equiv 0 \pmod{p}$ for $i\in \{2,\ldots,\mu\}$ and
  $c_1,\ldots,c_\mu \in \langle b_1,\ldots, b_r \rangle \setminus
  \{1\}$, and distinguish two cases.  To increase the readability of
  exponents we use at times also the notation $s(i) = s_i$.

  \smallskip

  \noindent \underbar{Case 1:} $s_{\mu+1} \equiv 0 \pmod{p}$.
  Expressing
  \[
  \Theta_2(z)=[a,z_{n+1}\cdots z_p] =[a,a^{s_1} c_1 a^{s_2} \cdots
  a^{s_\mu} c_\mu]
  \]
  as a product of conjugates of the $c_i^{\pm 1}$ by powers of $a$ and
  relabelling the $c_i^{\pm 1}$ as $\overline{c}_j$ for $j \in \{1,
  \ldots, m \}$, we get
  \begin{equation} \label{eq:4.10} \Theta_2(z) = \overline{c}_1^{\,a}
    \cdot \overline{c}_2^{\,a^{1+s(\mu)}} \cdot \ldots \cdot
    \overline{c}_{\mu}^{\,a^{1+ s(\mu)+ \ldots + s(2)}} \cdot
    \overline{c}_{\mu +1}^{\, a^{s(\mu)+ \ldots + s(2)}} \cdot \ldots
    \cdot \overline{c}_{m-1}^{\,a^{s(\mu)}} \cdot \overline{c}_m.
  \end{equation}
  Consider now
  \[
  \psi_1(\Theta_2(z)) = ((\Theta_2(z))_1,\ldots,(\Theta_2(z))_p).
  \]
  If $\partial((\Theta_2(z))_1) < \mu$ then
  $\Theta_1(\Theta_2(z))$ has length less than $m$.  Hence we suppose
  \[
  \partial((\Theta_2(z))_1) =\mu.
  \]

  Using the symbol $*$ for unspecified exponents, we deduce
  from~\eqref{eq:4.10} that the first components
  $(\overline{c}_j^{\,a^*})_1$ for odd $j\in \{1,\ldots,m-1\}$ must be
  non-trivial elements of $\langle b_1,\ldots,b_r \rangle$, and the
  $(\overline{c}_j^{\,a^*})_1$ for even $j \in \{2,\ldots,m-2\}$ must
  be non-trivial powers of $a$.  In particular, looking at the
  $(m-1)$th term we require $s_\mu \equiv 1 \pmod{p}$.  This implies
  that the second factor in~\eqref{eq:4.10} is
  $\overline{c}_2^{\,a^2}$.

  In the special case $n=1$, $e_{i,p-1}\equiv 0$ (mod $p$) for 
  every $i\in \{1,\ldots,r\}$ and so we immediately get a
  contradiction, because $(\overline{c}_2^{\,a^2})_1$ contributes a
  trivial factor $1$ to $(\Theta_2(z))_1$ instead of a non-trivial
  power of $a$.

  In the generic case $n \geq 2$ we claim
  $\partial((\Theta_2(z))_{n+1}\cdots (\Theta_2(z))_p ) <\mu$, leading
  to $\partial(\Theta_2(\Theta_2(z))) < m$.  Indeed, only factors
  $\overline{c}_j^{\,a^*}$ in~\eqref{eq:4.10} for even $j \in
  \{2,\ldots,m\}$ can contribute non-trivial elements of $\langle
  b_1,\ldots,b_r \rangle$ to the product $(\Theta_2(z))_{n+1}\cdots
  (\Theta_2(z))_p$.  But since $n \geq 2$, the second factor
  $\overline{c}_2^{\,a^2}$ in~\eqref{eq:4.10} contributes only a power
  of~$a$.

  \medbreak

  \noindent\underbar{Case 2:} $s_{\mu+1} \not \equiv 0 \pmod{p}$.
  Similarly as in Case 1, we write
  \begin{equation} \label{eq:4.11} \Theta_2(z) =
    \overline{c}_1^{\,a^{1+s(\mu+1)}} \cdot
    \overline{c}_2^{\,a^{1+s(\mu+1)+s(\mu)}} \cdot \ldots \cdot    
    \overline{c}_{m-1}^{\,a^{s(\mu+1)+s(\mu)}} \cdot
    \overline{c}_m^{\,a^{s(\mu+1)}},
  \end{equation}
  where the $c_i^{\pm 1}$ are relabelled as $\overline{c}_j$ for $j \in
  \{1, \ldots, m \}$.  As before, it suffices to show that
  $\partial((\Theta_2(z))_1)<\mu$ or
  $\partial((\Theta_2(z))_{n+1} \cdots (\Theta_2(z))_p )<\mu$.

  Suppose $\partial((\Theta_2(z))_1) = \mu$. Then either: 

  \smallbreak

  \noindent (i) $(\overline{c}_j^{\,a^*})_1$ for odd $j\in
  \{1,\ldots,m-1\}$ is a non-trivial element of $\langle
  b_1,\ldots,b_r \rangle$, and $(\overline{c}_j^{\,a^*})_1$ for even
  $j\in \{2,\ldots,m-2\}$ is a non-trivial power of $a$; or

  \smallbreak

  \noindent (ii) $(\overline{c}_j^{\,a^*})_1$ for even $j\in
  \{2,\ldots,m\}$ is a non-trivial element of $\langle b_1,\ldots,b_r
  \rangle$, and $(\overline{c}_j^{\,a^*})_1$ for odd $j\in
  \{3,\ldots,m-1\}$ is a non-trivial power of $a$.

  \smallbreak

  In case (i), we deduce from $(m-1)$th term in \eqref{eq:4.11} that
  \[
  s_{\mu+1} + s_\mu \equiv 1 \pmod{p},
  \]
  and the second term in \eqref{eq:4.11} is equal to
  $\overline{c}_2^{\,a^2}$.  We may argue as in Case~1 that
  $\partial((\Theta_2(z))_{n+1} \cdots (\Theta_2(z))_p )<\mu$ so that
  $\partial(\Theta_2(\Theta_2(z))) < m$.

  In case (ii), we deduce from the $m$th term in \eqref{eq:4.11} that
  \[
  s_{\mu+1} \equiv 1 \pmod{p},
  \]
  and the first term in \eqref{eq:4.11} is $\overline{c}_1^{a^2}$.  In
  the generic situation $n\geq2$ we argue similarly as in Case~1 that
  $\partial((\Theta_2(z))_{n+1} \cdots (\Theta_2(z))_p )<\mu$ so that
  $\partial(\Theta_2(\Theta_2(z))) < m$.  It remains to deal with the
  special situation $n=1$, which makes use of the fact that the
  defining vectors satisfy $e_{i,p-1}\equiv 0$ for every $i\in
  \{1,\ldots, r\}$.  For $m\ge 6$ the argument follows as before.  For
  $m=4$, proceeding similarly, we obtain $\Theta_2(z)=
  \overline{c}_1^{\,a^2} \overline{c}_2^{\,a} \, \overline{c}_3 \,
  \overline{c}_4^{\,a}$, so $(\Theta_2(z))_1= b a^w c$ for some
  $b,c\in \langle b_1,\ldots,b_r \rangle$ and $w\in
  \mathbb{Z}/p\mathbb{Z}$.  Thus subject to relabelling,
  \begin{equation} \label{eq:m4}
    \Theta_1(\Theta_2(z))=\overline{c}_1^{\,a}
    \overline{c}_2^{\,a^{1-w}} \overline{c}_3^{\, a^{-w}}
    \overline{c}_4.
  \end{equation}
  As before, for $\partial(\Theta_1(\Theta_2(z))_1)=2$, we need
  $(\overline{c}_3^{\,a^*})_1$ to be a non-trivial element of $\langle
  b_1,\ldots,b_r \rangle$ and $(\overline{c}_2^{\, a^*})_1$ to be a
  non-trivial power of $a$. Looking at the third term of
  \eqref{eq:m4}, we require $w=-1$. However, then $\overline{c}_2^{\,
    a^2}$ contributes a trivial factor 1 to $\Theta_1(\Theta_2(z))_1$
  instead of a non-trivial power of $a$.  Hence we see that the length
  decreases, as required.
\end{proof}


\section{Maximal subgroups}
In the present section we prove Theorem~\ref{start} about maximal
subgroups of torsion multi-edge spinal groups.  As in \cite{Pervova4},
it is convenient to phrase part of the argument in terms of proper
dense subgroups with respect to the profinite topology.

\subsection{Dense subgroups}
We recall that the cosets of finite-index subgroups of a group $G$
form a base for the \textit{profinite topology} on $G$.  The group $G$
contains maximal subgroups of infinite index if and only if it
contains proper dense subgroups with respect to the profinite
topology.  Indeed, a subgroup $H$ of $G$ is \textit{dense} with
respect to the profinite topology if and only if $G = NH$ for every
finite-index normal subgroup $N$ of~$G$.  Therefore every maximal
subgroup of infinite index in $G$ is dense and every proper dense
subgroup is contained in a maximal subgroup of infinite index.

For the rest of the section, we fix a just infinite multi-edge spinal
group $G = G_\mathbf{E} = \langle a,b_1,\ldots,b_r\rangle$ acting on
the regular $p$-adic rooted tree $T$, for an odd prime $p$. Here
$\mathbf{E}$ is the $r$-tuple of defining vectors $\mathbf{e}_i =
(e_{i,1},\ldots,e_{i,p-1})$, for $i\in \{1,\ldots,r\}$.  For a vertex
$u$ of $T$, we write 
\[
G_u = U_u^G \cong G
\]
to denote the upper companion group at the vertex~$u$.  Every subgroup
$M$ of $G$ gives rise to a subgroup 
\[
M_u = U_u^M \leq G_u \cong G.
\]

\begin{proposition}[{\cite[Proposition 3.2]{Pervova4}}] \label{3.3.3}
  Let $T$ be a spherically homogeneous rooted tree and let $G \le
  \Aut(T)$ be a just infinite group acting transitively on each layer
  of~$T$. Let $M$ be a dense subgroup of $G$ with respect to the
  profinite topology. Then
  \begin{enumerate}
  \item[(1)] the subgroup $M$ acts transitively on each layer of the
    tree $T$,
  \item[(2)] for every vertex $u \in T$, the subgroup $M_u$ is dense
    in $G_u$ with respect to the profinite topology.
  \end{enumerate}
\end{proposition}

The next result extends~\cite[Lemma~3.3]{Pervova4}, which addresses
just infinite GGS-groups. Here we give a different and shorter proof
for just infinite multi-edge spinal groups.

\begin{proposition} \label{proper} Let $G$ be a just infinite
  multi-edge spinal group.  Let $M$ be a proper dense subgroup of $G$,
  with respect to the profinite topology. Then $M_u$ is a proper
  subgroup of $G_u$ for every vertex $u$ of~$T$.
\end{proposition}

\begin{proof}
  Assume on the contrary, that there exists a vertex $u$ of $T$ such
  that $M_u = G_u$.  Let $u$ be a vertex of minimal length $n$ with
  the specified property, and suppose $u=wx$ where $|w|=n-1$.  By
  Proposition~\ref{3.3.3} and induction, $M_w$ is a proper dense
  subgroup of $G_w$.  Since $G_w$ is isomorphic to $G$, we have
  $|u|=1$, say $u=u_1$ among the vertices $u_1, \ldots, u_p$ at
  level~$1$.

  Let $R = \rstab_M(u)_u$.  By our assumption, we have $R
  \trianglelefteq M_u=G_u$.  Since $G_u \cong G$ is just
  infinite, either $R$ has finite index in $G_u$, or $R$ is trivial.

  Suppose first that $R$ has finite index in $G_u$. Then
  \begin{align*}
    |G:M| & \leq |G:\rstab_M(1)| = |G:\stab_G(1)| \, |\stab_G(1):\rstab_M(1)| \\
    & \leq |G:\stab_G(1)| \, \left| \prod\nolimits_{i=1}^p G_{u_i} :
      \prod\nolimits_{i=1}^p \rstab_M(u_i)_{u_i} \right| \\
    & \leq |G:\stab_G(1)| \, |G_u : R|^p
  \end{align*}
  is finite.  But, being a proper dense subgroup, $M$ has infinite
  index in $G$.

  Hence $R$ is trivial, and so $\rstab_M(1)$ is trivial.  Thus
  \[
  |G/\rstab_G(1)| \geq |M/\rstab_M(1)| = |M|
  \]
  is infinite.  

  By Proposition~\ref{prop:special-group}, the group $G$ is not
  $\Aut(T)$-conjugate to the GGS-group $\mathcal{G}$
  in~\eqref{eq:exception}.  Hence Proposition~\ref{prop:branch} shows
  that $G$ is a branch group.  Thus $\rstab_G(1)$ has finite index in
  $G$, a contradiction.
\end{proof}

\begin{proposition}\label{proposition: first}
  Let $G= \langle a,b_1,\ldots,b_r\rangle$ be a just infinite
  multi-edge spinal group and $M$ a dense torsion subgroup of $G$,
  with respect to the profinite topology.  For each $i\in \{1,\ldots,r
  \}$, there is a vertex $u$ of $T$ and an element $g\in \stab _G(u)$
  acting on $T_u$ as $a^k$ for some $k\in \mathbb{Z}/p\mathbb{Z}$
  under the natural identification of $T_u$ and $T$, such that 
 \begin{enumerate}
 \item[(i)] $(M^g)_u=(M_u)^{a^k}$ is a dense torsion subgroup of $G_u
   \cong G$, 
 \item[(ii)] there exists $b\in \langle b_1,\ldots,b_r \rangle$ with
   $\varepsilon_{b_i}(b)\ne 0$ such that $b\in (M_u)^{a^k}$.
 \end{enumerate}
\end{proposition}

\begin{proof}
  Clearly, it suffices to prove the claim for $i=1$.  Assuming that
  there are $u$ and $k$ such that (ii) holds, then as $G_u \cong G$
  there exists $g\in \stab_G(u)$ satisfying (i). Hence it suffices to
  show that such $u$ and $k$ exist.
  
  Since $|G : G'|$
  is finite, $G'$ is open in the profinite topology.  Thus we find $x
  \in M \cap b_1 G'$.  In particular, $x \in \stab_G(1)$ with
  $\epsilon_{b_1}(x) \not \equiv 0 \pmod{p}$.  We argue by induction
  on $\partial(x) \geq 1$.

  First suppose that $\partial(x) = 1$.  Then $x$ has the form $x =
  b^{a^k}$, where $b \in \langle b_1,\ldots , b_r \rangle$ with
  $\epsilon_{b_1}(b) \not \equiv 0 \pmod{p}$ and $k \in \{
  0,1,\ldots,p-1 \}$.  Thus choosing the vertex $u$ to be the root of
  the tree $T$, we have $b \in (M_u)^{a^{-k}}$.

  Now suppose that $\partial(x) \ge 2$.  Recall from
  \eqref{eq:def-eps-hat} and \eqref{eq:def-eps} the definition of
  $\epsilon_{b_1} (x)$, and from \eqref{eq:def-phi} the definition of
  the maps $\varphi_j \colon \stab_G(1) \rightarrow G_{u_j}$, where
  $u_1, \ldots, u_p$ denote the first level vertices of~$T$.  For any
  vertex $u$ of $T$, the subtree $T_u$ has a natural identification
  with~$T$ and $G_u \cong G$.  We freely use the symbols $a, b_1,
  \ldots, b_r$ to denote also automorphisms of $T_u$ under this
  identification.  We claim that
  \begin{equation}\label{eq:1}
    \epsilon_{b_1}(\varphi_1 (x)) + \ldots + \epsilon_{b_1}(\varphi_p
    (x)) \equiv \epsilon_{b_1} (x) \not \equiv 0 \pmod{p}. 
  \end{equation}
  To see this, write $x$ as a product of conjugates $b_i^{a^*}$ of the
  directed generators $b_i$, $i \in \{1,\ldots,r\}$, by powers~$a^*$,
  where the symbol $*$ represents unspecified exponents.  Then
  $\epsilon_{b_1} (x)$ is the number of factors of the form
  $b_1^{a^*}$.  Each of these factors contributes a directed
  automorphism $b_1$ in a unique coordinate, and none of the other
  factors $b_2^{a^*}, \ldots, b_r^{a^*}$ contributes a $b_1$ in any of
  the coordinates. Hence \eqref{eq:1} holds.

  By~\eqref{eq:1}, there exists $j \in \{ 1,\ldots,p \}$ such that
  $\epsilon_{b_1} (\varphi_j(x))\not \equiv 0 \pmod{p}$.  Moreover,
  Lemma~\ref{shortening} shows that
  \begin{equation} \label{eq:phi-j-estimate}
    \partial(\varphi_j(x)) \le \lceil \partial(x) / 2 \rceil \le 
    (\partial(x) + 1)/2 < \partial(x).
  \end{equation}

  Suppose that $\tilde x = \varphi_j(x) \in M_{u_j}$ belongs to
  $\stab_{G_{u(j)}}(1)$, where we write $u(j) = u_j$ for
  readability. By Proposition~\ref{3.3.3}, the subgroup $M_{u_j}$ is
  dense in $G_{u_j} \cong G$.  Since $\epsilon_{b_1}(\tilde x) \not
  \equiv 0 \pmod{p}$ and $\partial(\tilde x) < \partial(x)$, the
  result follows by induction.

  Now suppose that $\varphi_j (x) \not \in \stab_{G_{u(j)}}(1)$.  For
  $\ell \in \{ 1, \ldots, p\}$ we claim that
  \begin{equation}\label{eq:2}
    \epsilon_{b_1}(\varphi_\ell (\varphi_j (x)^p)) \equiv
    \epsilon_{b_1} (\varphi_j (x)) \not \equiv 0 \pmod{p}.
  \end{equation}
  To see this, observe that $\varphi_j (x)$ is of the form
  \[
  \varphi_j(x) =a^k h ,
  \]
  for $k \not \equiv 0 \pmod{p}$ and $h \in \stab_{G_{u(j)}}(1)$ with
  $\psi_1(h)= (h_1,\ldots, h_p)$, say.  Hence raising $\varphi_j(x)$
  to the prime power $p$, we get
  \[
  \varphi_j(x)^p = (a^k h)^p = {h^{a^{(p-1)k}}} {h^{a^{(p-2)k}}}
  \cdots h^{a^k} h,
  \]
  and thus for $\ell \in \{ 1, \ldots, p\}$,
  \[
  \varphi_\ell (\varphi_j (x)^p) \equiv h_1 h_2 \cdots h_p
  \pmod{G_{u_{j\ell}}'}.
  \]
  Here $u_{j\ell}$ denotes the $\ell$th descendant of $u_j$.  Arguing
  similarly as for~\eqref{eq:1}, this yields
  \[
  \epsilon_{b_1}(\varphi_\ell(\varphi_j(x)^p)) \equiv
  \epsilon_{b_1}(h_1)+\ldots +\epsilon_{b_1}(h_p) \equiv
  \epsilon_{b_1}(h) \equiv \epsilon_{b_1}(\varphi_j(x)) \pmod{p}
  \]
  and \eqref{eq:2} holds.

  Furthermore, we claim that
  \begin{equation} \label{eq:got-shorter}
    \partial(\varphi_\ell (\varphi_j(x)^p)) \le \partial (\varphi_j
    (x)) < \partial(x). 
  \end{equation}
  The second inequality comes from~\eqref{eq:phi-j-estimate}.  To see
  that the first inequality holds, we note that
  \[
  \varphi_\ell (\varphi_j (x)^p) = \varphi_\ell (h^{a^{(p-1)k}})
  \cdots \varphi_\ell (h^{a^k}) \varphi_\ell (h),
  \]
  and $\partial(\varphi_j(x)) = \partial (h)$.  We write $h$ as a
  product of $\partial(h)$ conjugates $c_j^{a^*}$ of directed
  automorphisms $c_j \in \langle b_1,\ldots,b_r \rangle$, where the
  symbol $*$ represents unspecified exponents.  Each factor
  $c_j^{a^*}$ contributes a directed automorphism $c_j$ in a unique
  coordinate and powers of $a$ in all other coordinates.  Focusing on
  the $\ell$th coordinate, we can write $\varphi_\ell (\varphi_j
  (x)^p)$ as a product of powers of $a$ and the $\partial(h)$ directed
  automorphisms $c_j \in \langle b_1, \ldots, b_r \rangle$. Hence
  \eqref{eq:got-shorter} holds.

  If $\tilde x = \varphi_\ell(\varphi_j(x)^p) \in M_{u_{j \ell}}$
  belongs to $\stab_{G_{u(j \ell)}}(1)$, we argue as follows.  By
  Proposition~\ref{3.3.3}, the subgroup $M_{u_{j \ell}}$ is dense in
  $G_{u_{j \ell}} \cong G$.  Since $\epsilon_{b_1}(\tilde x) \not
  \equiv 0 \pmod{p}$ and $\partial(\tilde x) < \partial(x)$, the
  result follows by induction.
 
  In general, we apply the operation $y \mapsto \varphi_\ell(y^p)$
  more than once.  Since $M$ is a torsion group, $x \in \stab_M(1)$
  and $\varphi_j(x)$ have finite order.  Clearly, if $y \in G$ has
  finite order then $\varphi_\ell(y^p)$ has order strictly smaller
  than~$y$.  Thus after finitely many iterations, we inevitably reach
  an element 
  \[
  \tilde x = \varphi_\ell( \varphi_\ell( \cdots
  \varphi_\ell(\varphi_\ell(\varphi_j(x)^p)^p)^p \cdots)^p ) \in
  M_{u_{j \ell \cdots \ell}}
  \]
  which in addition to the inherited properties $\epsilon_{b_1}(\tilde
  x) \not \equiv 0 \pmod{p}$ and $\partial(\tilde x) < \partial(x)$
  satisfies $\tilde x \in \stab_{G_{u(j \ell \cdots \ell)}}(1)$.  As
  before, the proof concludes by induction.
\end{proof}

Recall the definition of the family $\boldsymbol{\mathcal{E}}$ of
groups by means of~\eqref{eq:H}.

\begin{proposition}\label{proposition: second}
  Let $G= \langle a,b_1,\ldots,b_r\rangle$ be a multi-edge spinal
  group. Suppose $G$ is not $\Aut(T)$-conjugate to a group in
  $\boldsymbol{\mathcal{E}}$.  Let $M$ be a dense subgroup of $G$, with respect to
  the profinite topology, and suppose that $b_1\in M$. Then there
  exist a vertex $u$ of $T$ and an element $g\in \stab_G(u)$ acting on
  $T_u$ as $d\in \langle b_1,\ldots,b_r \rangle$ under the natural
  identification of $T_u$ and $T$, such that
 \begin{enumerate}
 \item[(i)] $(M^g)_u=(M_u)^d$ is a dense subgroup of $G_u\cong G$, 
 \item[(ii)] $a,b_1\in(M_u)^d$.
 \end{enumerate}
\end{proposition}

\begin{proof}
  Akin to the proof of Proposition \ref{proposition: first}, it
  suffices to establish the existence of $u$ and $d$ such that (ii)
  holds. Observe that $G'$ is open in $G$.  Since $M$ is dense in $G$,
  there is $z \in G'$ such that $az \in M$.  Write
  $\psi_1(z)=(z_1,\ldots,z_p)$.

  Let $u_p$ denote the $p$th vertex at level~$1$.  The coordinate map
  $\varphi_p$ allows us to restrict $\stab_M(1)$ to $M_{u_p}$.
  Clearly, $b_1 \in M$ implies $b_1 = \varphi_p(b_1) \in M_{u_p}$.
  Based on Lemma~\ref{dagger}, we assume that the defining vector
  $\mathbf{e}_1$ for $b_1$ has first coordinate $e_{1,1}=1$.  Consider
  the theta maps $\Theta_1, \Theta_2$ defined in
  Section~\ref{sec:theta-def}, with reference to~$b_1$.  By their
  definition, $a \, \Theta_1(z)$ and $a \, \Theta_2(z)$ belong to
  $M_{u_p}$.  Moreover, repeated application of $\varphi_p$
  corresponds to repeated applications of $\Theta_1$ and $\Theta_2$.
  By Proposition~\ref{3.3.3} and Theorem~\ref{theta}, we may assume
  that $\partial(z) \in \{0,2\}$.  If $\partial(z)=0$ we are done (with $d=1$).

  Thus we may assume that $\partial(z)=2$ and we write $z = b^{-a^m}
  c^{a^k}$ for $b,c \in \langle b_1,\ldots,b_r \rangle \setminus
  \{1\}$ and $m, k \in \mathbb{Z}/p\mathbb{Z}$ with $m \ne k$.

  \smallskip

  \noindent \underline{Case 1}: $m,k \ne 1$.  Here $z_1 = a^w$ for
  some $w\in \mathbb{Z}/p\mathbb{Z}$. Thus $\Theta_1(z) = [a,z_1^{-1}]
  = [a,a^{-w}]=1$, and $a \in M_{u_p}$.

  \smallskip

  \noindent \underline{Case 2}: $m=1$, $k\ne 1$.  Here
  \[
  \psi_1(b^{-a}) = (b^{-1},*,\ldots,*) \quad \text{and} \quad
  \psi_1(c^{a^k}) = (a^w,*,\ldots,*),
  \]
  where $w \in \mathbb{Z}/p\mathbb{Z}$ and the symbols $*$ denote
  unspecified entries.  Hence $z_1 = b^{-1}a^w$ so that $\Theta_1(z) =
  [a,z_1^{-1}] = [a,b]$.  This gives $a\,\Theta_1(z) = b^{-1} a
  b$. Remembering that $b_1$ and $b$ commute, we obtain $a, b_1 \in
  (M_{u_p})^{b^{-1}}$.

  \smallskip

  \noindent \underline{Case 3}: $m \ne 1$, $k=1$.  Here
  \[
  \psi_1(b^{-a^m}) = (a^w,*,\ldots,*) \quad \text{and} \quad
  \psi_1(c^a) = (c,*,\ldots,*),
  \]
  where $w \in \mathbb{Z}/p\mathbb{Z}$ and the symbols $*$ denote
  unspecified entries.  Hence $z_1 = a^wc$ so that $\Theta_1(z) =
  [a,z_1^{-1}] = c^{a^{1-w}}c^{-a^{-w}}$.  If $w \not \equiv -1
  \pmod{p}$, we are back in Case~1 or Case~2.

  Suppose that $w \equiv -1 \pmod{p}$.  Then $\Theta_1(z) =
  c^{a^2}c^{-a}$, where
  \[
  c^{a^2} = (*,c,*,\ldots,*) \quad \text{and} \quad
  c^{-a}=(c^{-1},*,\ldots,*)
  \]
  and the symbols $*$ denote unspecified powers of~$a$.  We recall from
  the definition of $\Theta_2$ that in the generic case $n \geq 2$
  this gives $\Theta_2(\Theta_1(z)) = 1$, hence $a, b_1\in
  M_{u_{pp}}$, where $u_{pp}$ is the $p^2$th vertex at
  level~$2$.  In the special case $n=1$ we have  $\Theta_1(\Theta_1(z)) =
  [a,c]$.  In this case we proceed similarly as in Case~2.
\end{proof}

\begin{proposition} \label{isG} Let $G$ be a just infinite multi-edge
  spinal group. Suppose $G$ is not $\Aut(T)$-conjugate to a group in
  $\boldsymbol{\mathcal{E}}$. Let $M$ be a dense torsion subgroup of
  $G$, with respect to the profinite topology. Then there exists a
  vertex $u$ of $T$ such that $M_u = G_u \cong G$.
\end{proposition}

\begin{proof}
We first remark that for a vertex $u$ of $T$ and $g\in G$, 
\begin{equation}\label{eq:conjugateaway}
  (M^g)_{u^g} = G_{u^g} \Longleftrightarrow M_u = G_u.
\end{equation}

Now by Proposition~\ref{proposition: first}, there exists a vertex
$u_1$ of $T$ and an element $g\in \stab_G(u_1)$ such that $x_1 \in (M^g)_{u_1}$
with $x_1 \in \langle b_1, \ldots, b_r \rangle$ and
$\varepsilon_{b_1}(x_1) \ne 0$.  We modify our generating set of
directed automorphisms, by taking $\tilde{b}_1 = x_1$ instead
of~$b_1$.  Using~\eqref{eq:conjugateaway} we may assume, without loss
of generality, that $\tilde{b}_1\in M_{u_1}$.

By Proposition~\ref{proposition: second}, there exists a vertex $u_1v$
of $T_{u_1}$ and an element $h\in \stab_G(u_1v)$ with $a,\tilde{b}_1 \in
(M^h)_{u_1v}$.  Once more by \eqref{eq:conjugateaway}, we may assume
that $a,\tilde{b}_1\in M_{u_1v}$.
 
Applying Proposition~\ref{proposition: first} again, we see that there
exists $k\in \mathbb{Z}/p\mathbb{Z}$ and a vertex $u_1vu_2$ of
$T_{u_1v}$ such that $x_2 \in (M_{u_1vu_2})^{a^k}$ with $x_2\in
\langle \tilde{b}_1,b_2, \ldots, b_r\rangle $ and
$\varepsilon_{b_2}(x_2)\ne 0$. Note that $\varepsilon_{b_2}$ is now
defined with respect to the new generating set
$\tilde{b}_1,b_2,\ldots,b_r$ of directed automorphisms. Since
$a,\tilde{b}_1\in M_{u_1v}$, it follows that $a,\tilde{b}_1\in
M_{u_1vu_2}$; recall Proposition~\ref{transitive} with respect to the
multi-edge spinal group $\langle a, \tilde{b}_1 \rangle$.  In
particular, $(M_{u_1vu_2})^{a^k}=M_{u_1vu_2}$.  Hence, replacing $b_2$
by $\tilde{b}_2=x_2$, we obtain $a,\tilde{b}_1,\tilde{b}_2 \in
M_{u_1vu_2}$.  Continuing in this manner, we arrive at a vertex $u =
u_1vu_2 \cdots u_r$ such that $a,\tilde{b}_1,\ldots , \tilde{b}_r \in
M_u$, where $a,\tilde{b}_1,\ldots, \tilde{b}_r$ is a generating set
for~$G_u$.
\end{proof}

\begin{theorem} \label{last} Let $G=\langle a,b_1,\ldots,b_r \rangle$
  be a just infinite multi-edge spinal group. Suppose $G$ is not
  $\Aut(T)$-conjugate to a group in $\boldsymbol{\mathcal{E}}$. Then
  $G$ does not contain any proper dense torsion subgroups, with
  respect to the profinite topology.
\end{theorem}

\begin{proof}
  Suppose on the contrary that $M$ is a proper dense torsion subgroup of $G$,
  with respect to the profinite topology. By Proposition~\ref{proper},
  for every vertex $u\in T$ we have $M_u$ is properly contained in
  $G_u$. However, by Proposition~\ref{isG}, the subgroup $M_u$ is all
  of $G\cong G_u$. This gives the required contradiction.
\end{proof}

Since $\boldsymbol{\mathcal{E}}$ does not contain any torsion groups
and every torsion multi-edge spinal group is a $p$-group,
Theorem~\ref{start} is a direct consequence of Theorem~\ref{last}.

We finish by proving Corollary \ref{cor:main}.

\begin{proof} [Proof of Corollary \ref{cor:main}]
  Suppose that $G = G_\mathbf{E} = \langle a, b_1,\ldots, b_r
  \rangle$, the multi-edge spinal group associated to some defining
  vectors~$\mathbf{E}$.  By Theorem~1.2, all maximal subgroups of $G$
  have finite index.  Since $G$ is residually finite and just
  infinite, its chief factors are finite.  Hence
  \cite[Lemma~3]{GrigWils} shows that, if $K$ is a subdirect subgroup
  of some direct product $G \times \ldots \times G$ of copies of~$G$,
  then all maximal subgroups of $K$ are of finite index.

  Let $H$ be a group commensurable with~$G$, and fix a finite index
  subgroup $L$ of~$H$ that is isomorphic to a finite index subgroup
  $K$ of~$G$.  For $i \in \{1,\ldots, r\}$, we write $G_i = \langle a,
  b_i \rangle$ for the GGS-subgroup of~$G$, generated by $a$
  and~$b_i$.  By \cite{Pervova1}, each of these GGS-groups $G_i$ has the congruence
  subgroup property.  Since $K \cap G_i$ has finite index in $G_i$,
  there is $m_i \in \mathbb{N}$ such that
  \[
  \mathrm{Stab}_{G_i}(m_i) \subseteq K \cap G_i \subseteq K.
  \]
  Since $\mathrm{Stab}_{G_i}(m_i)$ is subdirect in $G_i \times
  \overset{p^{m_i}}{\ldots} \times G_i$, we conclude that, for $m =
  \max \{ m_i \mid 1 \leq i \leq r \}$, the group $\mathrm{Stab}_K(m)$
  is subdirect in $G \times \overset{p^m}{\ldots} \times G$.

  As observed above, this implies that all maximal subgroups of
  $\mathrm{Stab}_K(m)$ are of finite index.  Since $\mathrm{Stab}_K(m)$
  has finite index in~$K$ and $L \cong K$ has finite index in~$H$, we
  deduce from~\cite[Lemma 1]{GrigWils} that all maximal subgroups of
  $H$ have finite index.
\end{proof}

\subsection*{Acknowledgement}
Special cases of the results in this paper form part of Alexoudas' PhD
thesis~\cite{thesis}.


\end{document}